\documentclass[a4paper,11pt]{article}
\usepackage[latin1]{inputenc}
\usepackage[english]{babel}
\usepackage[T1]{fontenc}
\usepackage{amsfonts,amssymb,amsthm,amsmath,amsbsy}
\usepackage{mathrsfs}
\usepackage{bm}
\usepackage{enumerate}
\usepackage{fullpage}
\usepackage{indentfirst}
\usepackage[affil-it]{authblk}

\newtheorem{teor}{Theorem}[section]
\newtheorem*{teor*}{Theorem}
\newtheorem{lemma}[teor]{Lemma}
\newtheorem{prop}[teor]{Proposition}
\newtheorem{corol}[teor]{Corollary}

\theoremstyle{definition}
\newtheorem{defin}[teor]{Definition}
\theoremstyle{remark}
\newtheorem{rmk}[teor]{Remark}
\newtheorem*{ex}{Example}

\newcommand{\Z}{\mathbb{Z}}
\newcommand{\N}{\mathbb{N}}
\newcommand{\C}{\mathbb{C}}
\newcommand{\End}{\mathrm{End}}

\newcommand{\Hom}{\mathrm{Hom}}
\newcommand{\ind}{\mathrm{ind}}
\newcommand{\inv}{\mathrm{inv}}
\newcommand{\ent}{\mathcal{O}}
\newcommand{\bs}{\backslash}
\newcommand{\frack}{\mathfrak{k}}
\newcommand{\f}{\mathbf{f}}

\numberwithin{equation}{section}

\newcommand{\keywords}[1]{\noindent \textbf{Keywords:}\quad #1}
\newcommand{\msc}[1]{\textbf{2010 Mathematics Subject Classification:}\quad #1}
\newcommand\blfootnote[1]{\begingroup\renewcommand\thefootnote{}\footnote{#1}\addtocounter{footnote}{-1}\endgroup}

\begin{document}
\title{Hecke algebra with respect to the pro-$p$-radical of a maximal compact open subgroup for $GL(n,F)$ and its inner forms}
\author{Gianmarco Chinello\thanks{Electronic address: \texttt{gianmarco.chinello[at]unimib.it}}}
\affil{Università degli Studi di Milano-Bicocca \\Dipartimento di Matematica e Applicazioni \\via Cozzi 55, 20125 Milano (Italy)}
\date{}

\maketitle

\begin{abstract}
\small\noindent 
Let $G$ be a direct product of inner forms of general linear groups over non-archimedean locally compact fields of residue characteristic $p$
and let $K^1$ be the pro-$p$-radical of a maximal compact open subgroup of $G$.
In this paper we describe the (intertwining) Hecke algebra $\mathscr{H}(G,K^1)$, that is the convolution 
$\Z$-algebra of functions from $G$ to $\Z$ that are bi-invariant for $K^1$ and whose supports are a finite union of $K^1$-double cosets. 
We produce a presentation by generators and relations of this algebra. 
Finally we prove that the level-$0$ subcategory of the category of smooth representations of $G$ over a unitary commutative ring $R$ such that $p\in R^{\times}$ is equivalent to the category of modules over $\mathscr{H}(G,K^1)\otimes_\Z R$.
\end{abstract}

\blfootnote{\msc{20C08}}
\blfootnote{\keywords{Hecke algebras, Presentation by generators and relations, Modular representations of p-adic reductive groups, Level 0 representations, inner forms of $p$-adic general linear groups.}}

\section*{Introduction}
Let $r$ be a positive integer and $p$ be a prime number. 
For every $i\in\{1, \dots,r \}$, let $F_i$ be a non-archimedean locally compact field of residue characteristic $p$, $D_i$ be a central division algebra of finite dimension over $F_i$ whose reduced degree is denoted by $d_i$ and $m_i$ be a  positive integer. 
We consider the group $G=\prod_{i=1}^rGL_{m_i}(D_i)$ which is an inner form of $\prod_{i=1}^rGL_{m_id_i}(F_i)$. 
Let $K^1$ be the pro-$p$-radical of a maximal compact open subgroup of $G$. 

\smallskip
The main purpose of this paper is to describe the (intertwining) Hecke algebra $\mathscr{H}(G,K^1)$, that is the convolution $\Z$-algebra of functions $\Phi:G\longrightarrow \Z$ such that $\Phi(kgk')=\Phi(g)$ for every $k,k'\in K^1$ and $g\in G$ and whose supports are a finite union of $K^1$-double cosets. 
For $r=1$, this result is given  by definition \ref{defA}, where we define an algebra abstractly by generators and relations, and by corollary \ref{isomalgebre}, where we prove that this algebra is isomorphic to $\mathscr{H}(G,K^1)$. 
For $r>1$, the result follow by remark \ref{rmkprodottodiretto}.

\smallskip
We have worked in this generality, considering $\Z$ as the base ring, because for every unitary commutative ring $R$ the $R$-algebra $\mathscr{H}_{R}(G,K^1)$ of functions $\Phi:G\longrightarrow R$ satisfying the conditions above, is isomorphic to the algebra $\mathscr{H}(G,K^1)\otimes_\Z R$. 
In this way, we obtain a description of $\mathscr{H}_{R}(G,K^1)$ for every ring and in particular for every field; actually this description is also new for the $\C$-algebra $\mathscr{H}_{\C}(G,K^1)$.

\smallskip
Let now $r=1$ and $\ell$ be a prime number different from $p$. 
This paper is a first step in the attempt to describe blocks (indecomposable direct summands) in the Bernstein decomposition of the category 
$\mathscr{R}_R(G)$ of smooth $\ell$-modular representations of $G$ (see \cite{SeSt1} or \cite{Vig1} for the split case), i.e. representations of $G$ over an algebraically closed field $R$ of positive characteristic $\ell$.
In the case of complex representations, Bernstein \cite{Bern} found a block decomposition of $\mathscr{R}_\C(G)$ and in \cite{SeSt3} (or \cite{BK3} for the split case) it is proved that each block is Morita equivalent to a tensor product of algebras of type A. 
These algebras are related with the Iwahori-Hecke algebra $\mathscr{H}_{\C}(G,I)$ where $I$ is an Iwahori subgroup of $G$. 
In the case of $\ell$-modular representations, this construction of Morita equivalences does not hold and one of the problems that occurs is that the pro-order of $I$ can be divisible by $\ell$. 
Some partial results on descriptions of these algebras, which are Morita equivalent to blocks of $\mathscr{R}_R(GL_n(F))$, are given by Dat 
\cite{Dat3}, Helm \cite{Helm} and Guiraud \cite{Gui}.

\smallskip
The idea that justifies this paper is the following (some results which follow are contained in the first chapter of the Phd thesis \cite{Chin} of the author and they will be part of further papers). 
We replace $I$ by $K^1$, a pro-$p$-group, which has an invertible pro-order modulo $\ell$. 
After that we consider the level-$0$ subcategory $\mathscr{R}^0_R(G)$ of $\mathscr{R}_R(G)$ that is the smallest full subcategory which contains all irreducible representations of $G$ that admit non-zero $K^1$-invariant vectors. 
Thus $\mathscr{R}^0_R(G)$ results in a direct sum of blocks, called level-$0$ blocks, and in the last part of this paper (see corollary \ref{corolequivalenza}) we prove that it is Morita equivalent to $\mathscr{H}_R(G,K^1)$. 
Hence, in order to describe each level-$0$ block of $\mathscr{R}_R(G)$, it is sufficient to find the set $\mathscr{E}$ of primitive central idempotents of this algebra (see 2.5 and 2.6 of \cite{Chin}) and describe $e\mathscr{H}_R(G,K^1)$ for every $e\in \mathscr{E}$. 

\smallskip
Furthermore, \cite{Chin} contains a technique to describe all blocks of $\mathscr{R}_R(G)$, also those of positive level, using the description of the Hecke algebra that we give in this paper. 
We know (see \cite{BK2}) that in complex case at every block of $\mathscr{R}_\C(G)$ we can associate a pair $(J,\lambda)$, called a type, where 
$J$ is a compact open subgroup of $G$ and $\lambda$ is an irreducible representation of $J$, such that the block is Morita equivalent to the $\C$-algebra $\mathscr{H}_\C(G,\lambda)$ of endomorphisms of the compactly induced representation $\ind_J^G(\lambda)$.
As said before, in \cite{SeSt3} it is proved that this algebra is isomorphic to a tensor product of certain Iwahori-Hecke algebras. 
In the case of $\ell$-modular representations, this construction of Morita equivalences does not hold and, as in the level-$0$ case, one of the problems that occurs is that the pro-order of $J$ can be divisible by $\ell$. 
So the idea is the following: using the theory of semisimple supertypes (see \cite{MS, SeSt1}) we can take a pair $(J^1,\eta)$ where 
$J^1$ is a compact open pro-$p$-subgroup of $G$ and $\eta$ is an irreducible representation of $J^1$. 
Thus we can consider the direct sum of blocks of $\mathscr{R}_R(G)$ associated to this pair and, similarly to the level-$0$ case, we can easily prove that this direct sum is Morita equivalent to the algebra $\mathscr{H}_R(G,\eta)=\End_G(\ind_{J^1}^G(\eta))$ (see 4.1 of \cite{Chin}). 
Moreover, we can associate to $(J^1,\eta)$ a group $G'$, which is of the same type of $G$ (but in general with $r>1$), and the 
pro-$p$-radical $K'^1$ of a maximal compact open subgroup of $G'$. 
Now, thanks to the explicit presentation by generators and relations of $\mathscr{H}_R(G',K'^1)$ presented in this paper, 
in order to construct a homomorphism between $\mathscr{H}_R(G',K'^1)$ and $\mathscr{H}_R(G,\eta)$ we need only look for elements in 
$\mathscr{H}_R(G,\eta)$ satisfying all relations defining $\mathscr{H}_R(G',K'^1)$ (see 3.4 of \cite{Chin}). 
Finally, using some properties of $\eta$, it is easy to prove that this homomorphism must be an isomorphism and so we have an equivalence of categories between any block of $\mathscr{R}_R(G)$ with a certain level-$0$ block of $\mathscr{R}_R(G')$, obtaining a complete description of 
$\mathscr{R}_R(G)$. 

\smallskip
We now give a brief summary of the contents of each section of this paper. 
In section 1 we present general results on the intertwining Hecke algebras for a generic group.  
In section 2 we introduce the algebra $\mathscr{H}(G,K^1)$: first we reduce its description to the case when $r=1$, then we choose a list of generators (proposition \ref{propgeneratori}), we find some relations among these elements (definition \ref{defA}) and finally we prove that they give rise to a presentation of $\mathscr{H}(G,K^1)$ by generators and relations (corollary \ref{isomalgebre}). 
In section 3 we prove that the level-$0$ subcategory of $\mathscr{R}_R(G)$, where $R$ is a unitary commutative ring such that 
$p\in R^{\times}$, is equivalent to the category of right modules over $\mathscr{H}_R(G,K^1)$.

\subsection*{Notations}
Throughout this paper we use the following notations: 
we denote by $\N$ the set of natural numbers and by $\N^*$ the set of (strictly) positive integers.
If $X$ is a finite set we denote its cardinality by $|X|\in\N$.
We write $\bigsqcup$ for a disjoint union.
If $A$ is a unitary ring and $n\in\N^*$ we denote by $\mathbb{I}_n$ the identity matrix with coefficients in $A$ and if
$a_1,\dots,a_{m}\in A$ we denote by $\mathrm{diag}(a_1,\dots,a_{m})$ the diagonal matrix with diagonal entries $a_1,\dots,a_{m}$.

\section{Hecke algebras for a generic group}\label{sezionealgebratriviale}
This section is written in much more generality than the remainder of this paper. 
We present general results on Hecke algebras for a generic multiplicative group.
All results are contained in first chapter of \cite{Krieg}. 

\smallskip
Let $\mathtt{G}$ be a multiplicative group and let $\mathtt{H}$ be a subgroup of $\mathtt{G}$ such that every $\mathtt{H}$-double coset is a finite union of left $\mathtt{H}$-cosets 
(or equivalently $\mathtt{H}\cap g\mathtt{H}g^{-1}$ is of finite index in $\mathtt{H}$ for every $g\in \mathtt{G}$). 
Such a pair $(\mathtt{G},\mathtt{H})$ is called \emph{Hecke pair} and we can associate to it the following algebra. 

\begin{defin}
Let $\mathscr{H}(\mathtt{G},\mathtt{H})$ be the $\Z$-algebra of functions $\Phi:\mathtt{G}\longrightarrow \Z$ such that 
$\Phi(hgh')=\Phi(g)$ for every $h,h'\in \mathtt{H}$ and $g\in \mathtt{G}$ and whose supports are a finite union of $\mathtt{H}$-double cosets, 
endowed with convolution product
\begin{equation}\label{prodconvoluzione}
(\Phi_1*\Phi_2)(g)=\sum_{x}\Phi_1(x)\Phi_2(x^{-1}g)
\end{equation}
where $x$ describes a system of representatives of  $\mathtt{G}/\mathtt{H}$ in $\mathtt{G}$. 
This algebra is unitary and the identity element is the characteristic function of $\mathtt{H}$.
\end{defin}

We remark that the sum in (\ref{prodconvoluzione}) is finite since the support of $\Phi_1$ is a finite union of $\mathtt{H}$-double cosets and 
by hypothesis, every $\mathtt{H}$-double coset is a finite union of left $\mathtt{H}$-cosets. 
Moreover (\ref{prodconvoluzione}) is well defined because for every $h\in\mathtt{H}$ and $x,g\in\mathtt{G}$ we have
$\Phi_1(xh)\Phi_2((xh)^{-1}g)=\Phi_1(x)\Phi_2(x^{-1}g)$.

\begin{rmk}\label{remR}
If $R$ is any unitary commutative ring then there are a canonical isomorphism of algebras between 
$\mathscr{H}(\mathtt{G},\mathtt{H})\otimes_{\Z}R$ and the $R$-algebra $\mathscr{H}_R(\mathtt{G},\mathtt{H})$
of functions $\Phi:\mathtt{G}\longrightarrow R$ such that 
$\Phi(hgh')=\Phi(g)$ for every $h,h'\in \mathtt{H}$ and $g\in \mathtt{G}$ and whose support is a finite union of $\mathtt{H}$-double cosets, endowed with convolution product (\ref{prodconvoluzione}). 
\end{rmk}

For every  $x\in \mathtt{G}$ we denote by $f_x:\mathtt{G}\longrightarrow \Z$ the characteristic function of the double coset $\mathtt{H}x\mathtt{H}$ and we choose one times for all a set $\Xi$ of representatives of $\mathtt{H}$-double cosets of $\mathtt{G}$. 
Thus the set $\{f_{x}\,|\, x\in\Xi\}$ is a basis of $\mathscr{H}(\mathtt{G},\mathtt{H})$ as $\Z$-module
and every element of the algebra can be written as $\sum_{x\in\Xi}a_x f_{x}$ with $a_x\in\Z$ for every $x\in\Xi$.
To simplify notations from now on we denote $f_{x}f_{y}=f_{x}*f_{y}$ for all $x,y\in\mathtt{G}$.

\smallskip
Let $x,y\in \mathtt{G}$. 
We have 
$\mathtt{H}x\mathtt{H}y\mathtt{H}=\bigsqcup_{\xi\in \Upsilon_{xy}} \mathtt{H}\xi\mathtt{H}$
with $\Upsilon_{xy}\subset\Xi$ and for every $\xi\in \Upsilon_{xy}$ we have
$ \mathtt{H}x\mathtt{H}\cap \xi\mathtt{H}y^{-1}\mathtt{H}=\bigsqcup_{z\in Z_\xi}z\mathtt{H}$
where $Z_\xi$ is a finite subset of $\mathtt{G}$. 
By (\ref{prodconvoluzione}) we obtain that the support of $f_{x}f_{y}$ is contained in $\mathtt{H}x\mathtt{H}y\mathtt{H}$ and
 $(f_{x}f_{y})(\xi)=|Z_\xi|=|(\mathtt{H}x\mathtt{H}\cap \xi\mathtt{H}y^{-1}\mathtt{H})/\mathtt{H}|$  for every $\xi\in\Upsilon_{xy}$. 
This implies that
\begin{equation}\label{eqprodotto}
f_{x}f_{y}=\sum_{\xi\in \Upsilon_{xy}}|Z_\xi|f_\xi\,.
\end{equation}

\begin{lemma}\label{lemmaprodotto} Let $x,y\in \mathtt{G}$. 
The support of $f_{x}f_{y}$ is $\mathtt{H}x\mathtt{H}y\mathtt{H}$ and if 
$x$ or $y$ normalizes $\mathtt{H}$ then $f_{x}f_{y}=f_{xy}$.
\end{lemma}

\begin{proof}
Let $\xi\in\Upsilon_{xy}$. 
In order to prove the first assertion we have to prove that $|Z_\xi|>0$. 
We have $\xi=h_1xh_2yh_3$ with $h_1,h_2,h_3\in \mathtt{H}$ and then 
$h_1x\mathtt{H}\subset \mathtt{H}x\mathtt{H}\cap \xi\mathtt{H}y^{-1}\mathtt{H}$ which implies $|Z_\xi|>0$. 
Now, if $x$ or $y$ normalizes $\mathtt{H}$ then $\mathtt{H}x\mathtt{H}y\mathtt{H}=\mathtt{H}xy\mathtt{H}$, $\Upsilon_{xy}=\{\tilde \xi\}$ with 
$\mathtt{H}\tilde \xi \mathtt{H}=\mathtt{H}xy\mathtt{H}$ and $|Z_{\tilde \xi}|=1$. 
Hence we obtain $f_{x}f_{y}=f_{\tilde \xi}=f_{xy}$.
\end{proof}

\begin{rmk}\label{rmkconiugato}
Let $\mathtt{K}$ be any subgroup of $\mathtt{G}$ containing $\mathtt{H}$ and let $g\in \mathtt{G}$ such that 
$\mathtt{K}^g=g^{-1}\mathtt{K}g=\mathtt{K}$. 
Then $(\mathtt{K},\mathtt{H})$ and $(\mathtt{K},\mathtt{H}^g)$ are Hecke pairs and applying theorem 6.1 of \cite{Krieg} with 
$\varphi:\mathtt{K}\rightarrow \mathtt{K}$ given by $k\mapsto g^{-1}kg$ for every $k\in\mathtt{K}$, we obtain an isomorphism of algebras 
$\mathscr{H}(\mathtt{K},\mathtt{H})\cong\mathscr{H}(\mathtt{K},\mathtt{H}^g)$.
\end{rmk}

\begin{rmk}\label{rmkprodottodiretto}
Let $r\in\N^*$ and let $(\mathtt{G}_i,\mathtt{H}_i)$ be an Hecke pair for every $i\in\{1,\dots,r\}$. 
Then $(\mathtt{G}_1\times\cdots\times \mathtt{G}_r,\mathtt{H}_1\times\cdots\times \mathtt{H}_r)$ is an Hecke pair and 
the algebra $\mathscr{H}(\mathtt{G}_1\times\cdots\times \mathtt{G}_r,\mathtt{H}_1\times\cdots\times \mathtt{H}_r)$ 
is isomorphic to the tensor product
$\mathscr{H}(\mathtt{G}_1,\mathtt{H}_1)\otimes\cdots\otimes\mathscr{H}(\mathtt{G}_r,\mathtt{H}_r)$ (see theorem 6.3 of \cite{Krieg}).
\end{rmk}

\section{The algebra $\mathscr{H}(G,K^1)$}\label{sectionB} 
In this section we focus the study of the algebra $\mathscr{H}(\mathtt{G},\mathtt{H})$ in the case where $\mathtt{G}$ is a direct product of inner forms of general linear groups over non-archimedean locally compact fields.

\smallskip
Let $r\in \N^*$ and $p$ be a prime number. 
For every $i\in\{1, \dots, r\}$, let $F_i$ be a non-archimedean locally compact field of residue characteristic $p$ and let $D_i$ be a central division algebra of finite dimension over $F_i$ whose reduced degree is $d_i\in\N$. 
Given $m_i\in\N^*$ for every $i\in\{1, \dots, r\}$, we denote $G_i=GL_{m_i}(D_i)$ which is an inner form of $GL_{m_id_i}(F_i)$.
Let $K_i$ be a maximal compact open subgroup of $G_i$ and let $K_i^1$ be the pro-$p$-radical of $K_i$. 
We remark that every $K_i^1$-double coset of $G_i$ is a finite union of left $K_i^1$-cosets and so we can use notations of section 
\ref{sezionealgebratriviale} with $\mathtt{G}=G_i$ and $\mathtt{H}=K_i^1$. 
In this section we want to study the algebra $\mathscr{H}(G_1\times\cdots\times G_r,K_1^1\times\cdots\times K_r^1)$. 
Thanks to remark \ref{rmkprodottodiretto} this algebra is isomorphic to the tensor product 
$\mathscr{H}(G_1,K_1^1)\otimes\cdots\otimes\mathscr{H}(G_r,K_r^1)$ and so we can reduce to study the case when $r=1$.
So, from now on in this section, we consider $r=1$ and we denote $F=F_1$, $D=D_1$, $m=m_1$, $G=G_1=GL_m(D)$, $K=K_1$ and $K^1=K_1^1$.

\smallskip
We denote by $\ent_{D}$ the ring of integers of $D$, by $\mathbb{\varpi}$ a uniformizer of $\ent_{D}$, by $\wp_{D}=\mathbf{\varpi}\ent_{D}$ its prime ideal and by $\frack_{D}=\ent_{D}/\wp_{D}$ the residue field of 
$D$ whose cardinality is denoted by $q\in \N^*$. 
Since $K^1$ is conjugate to the first congruence subgroup of $GL_m(\ent_D)$, by remark \ref{rmkconiugato} we can choose 
$K=GL_m(\ent_D)$ and $K^1=\mathbb{I}_m+M_m(\wp_{D})$. 
This is what we assume from now on.

\smallskip
We recall that the identity element of $\mathscr{H}(G,K^1)$ is the characteristic function of $K^1$. 
As in section 
\ref{sezionealgebratriviale} we denote by $f_x$ the characteristic function of the double coset $K^1 xK^1$ for every  $x\in G$ and we choose  a set of representatives $\Xi$ of $K^1$-double cosets of $G$. 
We recall that for every $x,y\in G$ the support of $f_{x}*f_{y}=f_{x}f_{y}$ is $K^1xK^1yK^1=\bigsqcup_{\xi\in \Upsilon_{xy}} K^1yK^1$ with $\Upsilon_{xy}\subset\Xi$. 
Furthermore, we have 
\begin{equation}\label{eqprodotto2}
f_{x}f_{y}=\sum_{\xi\in \Upsilon_{xy}}\big|(K^1xK^1\cap \xi K^1y^{-1}K^1)/K^1\big|f_\xi
\end{equation}
and if $x$ or $y$ normalizes $K^1$ then $f_{x}f_{y}=f_{xy}$.

\subsection{Root system}\label{sezioneradici}
In this paragraph we introduce the root system of a general linear group. 
The classical reference for general case is chapter VI of \cite{Bourb1} (see also \cite{Morris}).

\smallskip
We denote by $\mathbf{\Phi}=\{\alpha_{ij}\;|\;1\leq i\neq j\leq m\}$ the set of roots of $\mathbf{GL}_m$ relative to torus of diagonal matrices.
We consider the set of positive roots $\mathbf{\Phi}^+=\{\alpha_{ij}\;|\;1\leq i<j\leq m\}$, the set of negative roots
$\mathbf{\Phi}^-=-\mathbf{\Phi}^+=\{\alpha_{ij}\;|\;1\leq j<i\leq m\}$ and the set of simple roots 
$\Sigma=\{\alpha_{i,i+1}\;|\;1\leq i\leq m-1\}$ relative to Borel subgroup of upper triangular matrices. 

\smallskip
For every $\alpha=\alpha_{i,i+1}\in\Sigma$ we write $s_{\alpha}=s_i$ for the transposition $(i,i+1)$. 
Let $W$ be the group generated by the $s_i$ which is the group of permutations of $m$ elements and so the Weyl group of $\mathbf{GL}_m$.
Let $\ell:W\rightarrow \N$ be the length function of $W$ relative to $s_1,\dots,s_{m-1}$, defined by
\[\ell(w) = \min\big\{n\in\N\,|\, w=s_{\alpha_1}\cdots s_{\alpha_n} \text{ with } \alpha_j\in\Sigma\big\}\]
for every $w\in W$. 
Then we have $\ell(ws_\alpha)-\ell(w)\in\{1,-1\}$ for every $w\in W$ and $\alpha\in\Sigma$.

\smallskip
The group $W$ acts on $\bm\Phi$ by $w\alpha_{ij}=\alpha_{w(i)w(j)}$. 
We denote 
\[N(w)=\bm\Phi^+\cap w^{-1}\bm\Phi^-=\{\alpha\in\bm\Phi^+\,|\, w\alpha\in\bm\Phi^-\}\] 
for every $w\in W$ and we remark that
$N(s_\alpha)=\{\alpha\}$ for every $\alpha\in\Sigma$.
\begin{lemma}\label{radici1}
For every $w\in W$ and $\alpha\in\Sigma$ we have
$\displaystyle{	
	|N(ws_\alpha)|=\left\{
	\begin{array}{ll}
	|N(w)|+1 &\text{ if } w\alpha\in\bm\Phi^+\\
	|N(w)|-1 &\text{ if } w\alpha\in\bm\Phi^-.
	\end{array}
	\right.}$
\end{lemma}

\begin{proof}
In the case $w\alpha\in\bm\Phi^+$ we have $N(ws_{\alpha})=s_\alpha N(w)\cup \{\alpha\}$. 
Otherwise if $w\alpha\in\bm\Phi^-$ we obtain 
$N(ws_{\alpha})=s_\alpha(N(w)\setminus\{\alpha\})$.
\end{proof}

\begin{prop}\label{propradici}
For every $w\in W$ we have $\ell(w)=|N(w)|$.
\end{prop}

\begin{proof} Note that $|N(1)|=\ell(1)=0$. 
We prove $N(w)\leq \ell(w)$ by induction on length of $w\in W$. 
Let $w,w'\in W$, $n\in\N$ and $\alpha\in\Sigma$ be such that $\ell(w)=n+1$, $\ell(w')=n$ and $w=w's_\alpha$. 
By induction hypothesis and by lemma \ref{radici1} we have 
$|N(w)|=|N(w's_\alpha)|\leq|N(w')|+1 \leq \ell(w')+1=\ell(w)$.
We prove $\ell(w)\leq N(w)$ by induction on $|N(w)|$. 
Let $w\neq 1$ be in $W$ and $n\in \N$ be such that $|N(w)|=n+1$. 
By lemma \ref{radici1} there exists $\alpha\in\Sigma$ such that $|N(ws_{\alpha})|=n$ otherwise $w\alpha'\in\bm\Phi^+$ for every 
$\alpha'\in\Sigma$ that implies $w(1)<w(2)<\cdots <w(m-1)$ and so $w=1$. 
By induction hypothesis we obtain 
$\ell(w)\leq\ell(ws_\alpha)+1\leq |N(ws_{\alpha})|+1=|N(w)|$.
\end{proof}

\noindent
Putting together lemma \ref{radici1} with proposition \ref{propradici}, we obtain 
\begin{equation}\label{radici2}
\ell(ws_\alpha)=\left\{
\begin{array}{ll}
\ell(w)+1 &\text{ if } w\alpha\in\bm\Phi^+\\
\ell(w)-1 &\text{ if } w\alpha\in\bm\Phi^-
\end{array}
\right.
\end{equation}
for every $w\in W$ and $\alpha\in\Sigma$.

\begin{lemma}\label{radici3}
For every $w_1,w_2\in W$ we have $N(w_1w_2)\subset N(w_2)\sqcup w_2^{-1}N(w_1)$. 
Moreover the equality holds if and only if $N(w_2)\subset N(w_1w_2)$
and if and only if $\ell(w_1w_2)= \ell(w_1)+\ell(w_2)$.
\end{lemma}

\begin{proof}
In order to prove first assertion we fix $\alpha \in N(w_1w_2)$. 
If $w_2\alpha\in\bm\Phi^-$ then $\alpha\in N(w_2)$ and if $w_2\alpha\in\bm\Phi^+$ then $w_2\alpha\in N(w_1)$ since $w_1w_2\alpha\in\bm\Phi^-$. 
Thus we have $N(w_1w_2)\subset N(w_2)\cup w_2^{-1}N(w_1)$. 
Moreover we have $N(w_2)\cap w_2^{-1}N(w_1)\subset w_2^{-1}(\bm\Phi^-\cap N(w_1))=\emptyset$. 
To prove the second assertion, we suppose $N(w_2)\subset N(w_1w_2)$ and we take $\alpha\in w_2^{-1}N(w_1)$. 
We have $w_1w_2\alpha\in\bm\Phi^-$ and so it remains to prove $\alpha\in \bm\Phi^+$.  
Since $w_1w_2(-\alpha)\in\bm\Phi^+$ we have $-\alpha\notin N(w_1w_2)\supset N(w_2)=\bm\Phi^+\cap w_2^{-1}\bm\Phi^-$
and then $\alpha\notin \bm\Phi^-\cap w_2^{-1}\bm\Phi^+\supset \bm\Phi^-\cap w_2^{-1}N(w_1)$.  
Since we have taken $\alpha\in w_2^{-1}N(w_1)$, $\alpha$ must be a positive root. 
Third assertion follows immediately from proposition \ref{propradici}.
\end{proof}

\noindent
In particular lemma \ref{radici3} implies that $\ell(w_1w_2)\leq \ell(w_1)+\ell(w_2)$ for every $w_1,w_2\in W$.

\smallskip
Let $P\subset\Sigma$. 
We denote by $\mathbf{\Phi}_P^+$ the set of positive roots generated by $P$, 
$\mathbf{\Phi}_P^-=-\mathbf{\Phi}_P^+$, 
$\mathbf{\Phi}_P=\mathbf{\Phi}_P^+\cup\mathbf{\Phi}_P^-$, 
$\mathbf{\Psi}_P^+=\mathbf{\Phi}^+\setminus\mathbf{\Phi}_P^+$ and
$\mathbf{\Psi}_P^-=-\mathbf{\Psi}_P^+$.
We denote by $W_P$ the subgroup of $W$ generated by the $s_{\alpha}$ with $\alpha\in P$. 
We denote $\widehat P=\Sigma-P$, $\widehat\alpha=\widehat{\{\alpha\}}$ and we observe that  
$\bm\Phi_P^+=\bigcap_{\alpha\in \widehat{P}}\bm\Phi_{\widehat\alpha}^+$ and 
$\bm\Psi_P^+=\bigcup_{\alpha\in \widehat{P}}\bm\Psi_{\widehat\alpha}^+$. 

\begin{ex}
Let $\alpha=\alpha_{i,i+1}\in\Sigma$. 
Then $\widehat\alpha=\{\alpha_{j,j+1}\in\Sigma\,|\,j\neq i\}$ and
\begin{align*}
\bm\Phi_{\widehat\alpha}^+&=\{\alpha_{hk}\in\bm\Phi^+\,|\,1\leq h<k\leq i \text{ or }i+1\leq h<k\leq m\},\\
\bm\Psi_{\widehat\alpha}^+&=\{\alpha_{hk}\in\bm\Phi^+\,|\,1\leq h\leq i < k\leq m \}.
\end{align*}
\end{ex}

\begin{prop}\label{proplongmin}
Let $P\subset\Sigma$ and $w$ be an element of minimal length in $wW_P\in W/W_P$.  
Then $w\alpha\in\mathbf{\Phi}^+$ for every $\alpha\in \mathbf{\Phi}_P^+$. 
\end{prop}

\begin{proof}
By hypothesis, for every $\alpha\in P$ we have $\ell(ws_\alpha)=\ell(w)+1$ and so by (\ref{radici2}) we have $w\alpha\in \bf\Phi^+$. 
\end{proof}

\begin{lemma}\label{propunicradice}
Let $P\subset \Sigma$. 
If $w$ is an element of minimal length in $wW_P\in W/W_P$ then for every $w'\in W_P$ we have 
$\ell(ww')=\ell(w)+\ell(w')$. 
\end{lemma}

\begin{proof}
Thanks to Lemma \ref{radici3}, in order to prove first assertion we need to show $N(w')\subset N(ww')$. 
Let $\alpha\in N(w')$. 
Since $w'\in W_P$ then $\alpha$ must be in $\bm\Phi_P^+$ and so $w'\alpha\in\bm\Phi_P^-$. 
By proposition \ref{proplongmin} we obtain $ww'\alpha\in w\bm\Phi_P^-\subset \bm\Phi^-$ and then $\alpha\in N(ww')$. 
\end{proof}

\noindent
Lemma \ref{propunicradice} implies that if $P\subset \Sigma$ then in each class of $W/W_P$ there exists a unique element of minimal length.
The same holds in each class of $W_P\bs W$ because if $w$ is of minimal length in $W_Pw\in W_P\bs W$ then $w^{-1}$ is of minimal length in $w^{-1}W_P\in W/W_P$.

\begin{prop}\label{proplongmin2}
Let $P,Q\subset \Sigma$, $w$ be the element of minimal length in $W_Pw\in W_P\bs W$ and $w'$ be the element of minimal length in $wW_Q\in W/W_Q$. 
Then $w'$ is the element of minimal length in $W_Pw' \in W_P\bs W$. 
\end{prop}

\begin{proof}
Since $w\in w'W_Q$, there exists $w''\in W_Q$ such that $w=w'w''$ and by lemma \ref{propunicradice} we have 
$\ell(w)=\ell(w')+\ell(w'')$. 
We now suppose by contradiction that there exists $\alpha\in P$ such that $\ell(s_\alpha w')<\ell(w')$. 
We obtain 
$\ell(s_\alpha w)=\ell(s_\alpha w'w'')\leq \ell(s_\alpha w')+\ell(w'')<\ell(w')+\ell(w'')=\ell(w)$ that contradicts the fact that $w$ is of minimal length in $W_Pw$. 
\end{proof}

\subsection{Generators}\label{sezionegeneratori}
In this paragraph we look for a set of generators of the $\Z$-algebra $\mathscr{H}(G,K^1)$ of the form $f_x$ with $x\in G$. 

\smallskip
For every $\alpha=\alpha_{i,i+1}\in\Sigma$ we consider the matrix  
\[\tau_{\alpha}=\tau_i=\begin{pmatrix} \mathbb{I}_i &0\\ 0&\varpi\mathbb{I}_{m-i}\end{pmatrix}\]
in $ G$ and we denote $\tau_0=\varpi\mathbb{I}_{m}$ and $\tau_m=\mathbb{I}_{m}$. 
Let $\bm\Delta$ be the commutative monoid generated by $\tau_\alpha$ with $\alpha\in\Sigma$.
Then we can write every element $\tau\in\bm\Delta$ uniquely as $\tau=\prod_{\alpha\in\Sigma}\tau_\alpha^{i_\alpha}$ with $i_\alpha\in\N$ and uniquely as $\tau=\mathrm{diag}(1,\varpi^{a_1},\dots,\varpi^{a_{m-1}})$ with $0\leq a_1\leq\cdots\leq a_{m-1}$. 
We denote $P(\tau)=\{\alpha\in\Sigma\,|\,i_{\alpha}=0\}$ and if $P\subset \{0,\dots,m\}$ or if $P\subset\Sigma$ we write $\tau_P$ in place of $\prod_{x\in P}\tau_x$. 
We remark that if $P\subset\Sigma$ then $P(\tau_P)=\widehat{P}$. 

\smallskip
We denote $\Omega=K\cup\{\tau_0,\tau_0^{-1}\}\cup\{\tau_\alpha\,|\, \alpha\in\Sigma\}$ and 
$\bm\Omega=\{f_\omega\in \mathscr{H}( G,K^1) \,|\, \omega\in\Omega\}$. 
The set $\bm\Omega$ is finite because if $\omega\in K^1$ then $f_\omega=1$ and $K/K^1\cong GL_m(\frack_{D})$ is a finite group.
We want to prove that the algebra $\mathscr{H}(G,K^1)$ is generated by $\bm\Omega$.

\smallskip
In all this section we consider the following subgroups of $K$.
\begin{enumerate}[$\bullet$]\label{sottogruppiK}
\item For every $\alpha=\alpha_{ij}\in\bm\Phi$ we denote by $U_{\alpha}$ the subgroup of matrices $(a_{hk})\in K$ with $a_{hh}=1$ for every 
$h\in\{1,\dots,m\}$, $a_{ij}\in \ent_{D}$ and $a_{hk}=0$ if $h\neq k$ and $(h,k)\neq (i,j)$. 
\item For every $P\subset\Sigma$ we denote by $M_{P}$ the intersections with $K$ of the standard Levi subgroup associated to $P$ and by
$U_{P}^+$ (resp. $U_{P}^-$) the intersections with $K$ of the unipotent radical of upper (resp. lower) standard parabolic subgroups with Levi factor the standard Levi subgroup associated to $P$. 
To simplify notations we denote by $U=U_{\emptyset}$ (resp. $U^-=U_{\emptyset}^-$) the subgroup of $K$ of upper (resp. lower) unipotent matrices. 
Finally we denote $M_{P}^+=M_P\cap U$ and $M_{P}^-=M_P\cap U^-$. 
Thus we have
\[U_{P}^+=\prod_{\alpha\in\mathbf{\Psi}_P^+}U_{\alpha},
\qquad 
U_{P}^-=\prod_{\alpha\in\mathbf{\Psi}_P^-}U_{\alpha},
\qquad
M_{P}^+=\prod_{\alpha\in\mathbf{\Phi}_P^+}U_{\alpha},
\qquad 
M_{P}^-=\prod_{\alpha\in\mathbf{\Phi}_P^-}U_{\alpha}.\]
Furthermore, if $P_1\subset P_2\subset\Sigma$ then $U_{P_2}^+$ is a subgroup of $U_{P_1}^+$ and  $U_{P_2}^-$ a subgroup of $U_{P_1}^-$. 
\item We know that there exists a unique multiplicative section of the surjection $\ent_{D}^{\times}\longrightarrow \frack_{D}^{\times}$. 
Then we can canonically identify the group of diagonal matrices with coefficients in $\frack_{D}^{\times}$ to a subgroup $T$ of the group of diagonal matrices with coefficients in $\ent_{D}^{\times}$.
\item We denote by $I=K^1TU$ the standard Iwahori subgroup of $G$ and $I^1=K^1U$ its pro-$p$-radical.
\item We identify $s_\alpha$ to an element of $ G$ for every $\alpha\in \Sigma$ and we identify $W$ to a subgroup of $G$ by permutation matrices.
\end{enumerate}

\begin{rmk}\label{rmqnormalizzatore}
The group $K^1$ is normal in $K$, 
$W$ normalizes $T$, 
$T$ normalizes every $U_{\alpha}$ with $\alpha\in\bm\Phi$,
$\tau_0$ centralizes $W$ and it normalizes $K^1$, $T$ and every $U_\alpha$ with $\alpha\in\bm\Phi$
and finally every $\tau\in\bm\Delta$ centralizes $W_{P(\tau)}$ and it normalizes $T$, $M_{P(\tau)}$ 
and every $U_{\alpha'}$ with $\alpha'\in \mathbf{\Phi}_{P(\tau)}$.
Note that in the case when $D=F$ the element $\tau_0$ is in the centre of $G$ and every $\tau\in\bm\Delta$ centralizes $M_{P(\tau)}$. 
\end{rmk}

\begin{rmk}\label{rmqlunghezza}
The group $W$ acts on the set of $U_\alpha$ with $\alpha\in\bm\Phi$ by $wU_\alpha w^{-1}=U_{w\alpha}$ and so by proposition \ref{propradici} we have $|(U^-\cap wU w^{-1})K^1/K^1|=|(U^-\cap wU w^{-1})/(K^1\cap U^-\cap wU w^{-1})|=q^{\ell(w)}$. 
\end{rmk}

We now state a lemma that is the basis for following calculations and that heavily use the fact that we are in $GL_m(D)$ and not in an another linear group.

\begin{lemma}\label{rmqtau1}
Let $\tau\in\bm\Delta$. 
Then we have $\tau^{-1} U_{P(\tau)}^+ \tau\subset U_{P(\tau)}^+\cap K^1$ and $\tau U_{P(\tau)}^- \tau^{-1}\subset U_{P(\tau)}^-\cap K^1$.
Furthermore, if $\tau=\tau_\alpha$ with $\alpha\in\Sigma$ then these inclusions are equalities.
\end{lemma}

\begin{proof}
We start with second assertion.
If $\alpha=\alpha_{i,i+1}\in\Sigma$ then $P(\tau_\alpha)=\widehat\alpha$ and
\[U_{\widehat\alpha}^+=\big\{(u_{hk})
\,|\, u_{hk}\in\ent_{D} \text{ if } 1\leq h\leq i < k\leq m, 
u_{hh}=1 \text{ if } 1\leq h\leq m \text{ and } u_{hk}=0 \text{ otherwise}\big\}\]
and so $\tau_\alpha^{-1} U_{\widehat\alpha}^+\tau_\alpha = U_{\widehat\alpha}^+\cap K^1$.
Similarly we can obtain $\tau_\alpha U_{\widehat\alpha}^-\tau_\alpha^{-1} = U_{\widehat\alpha}^-\cap K^1$.
For the general case we recall that 
$U_{P(\tau)}^+=\prod_{\alpha'\in\mathbf{\Psi}_{P(\tau)}^+}U_{\alpha'}$ and we fix $\alpha'\in \mathbf{\Psi}_{P(\tau)}^+$.
Since 
$\bm\Psi_{P(\tau)}^+=\bigcup_{\alpha\in \widehat{P(\tau)}}\bm\Psi_{\widehat{\alpha}}^+$, there exists 
$\alpha\in \widehat{P(\tau)}$ such that 
$\alpha'\in \bm\Psi_{\widehat{\alpha}}^+$ and so there exists
$\tau(\alpha)\in\bm\Delta$ such that $\tau=\tau_\alpha\tau(\alpha)$ and such that 
$\tau(\alpha)^{-1}(U_{\widehat\alpha}^+\cap K^1)\tau(\alpha)\subset U_{\widehat\alpha}^+\cap K^1$. 
We obtain 
$\tau^{-1} U_{\alpha'} \tau 
\subset \tau(\alpha)^{-1}(\tau_\alpha^{-1} U_{\widehat\alpha}^+\tau_\alpha)\tau(\alpha)
=\tau(\alpha)^{-1}(U_{\widehat\alpha}^+\cap K^1)\tau(\alpha)
\subset U_{\widehat\alpha}^+\cap K^1 $
that is contained in $U_{P(\tau)}^+\cap K^1$ because $P(\tau)\subset \widehat{\alpha}$. 
Hence we have
$\tau^{-1} U_{P(\tau)}^+ \tau\subset U_{P(\tau)}^+\cap K^1$. Similarly we can obtain $\tau U_{P(\tau)}^- \tau^{-1}\subset U_{P(\tau)}^-\cap K^1$.
\end{proof}

\begin{lemma}\label{relgruppi}
Let $\tau\in\bm\Delta$ and $\alpha\in\Sigma$.
\begin{enumerate}[$(a)$]
	\item We have $U_{P(\tau)}^+\tau K^1=\tau K^1$ and $K^1\tau U_{P(\tau)}^-=K^1\tau$.
	\item We have $\tau^{-1}K^1\tau K^1\cap K= U_{P(\tau)}^-K^1$ and $\tau K^1\tau^{-1} K^1\cap K= U_{P(\tau)}^+ K^1$.
	\item $K^1\tau_{\alpha}K^1=\tau_{\alpha}U_{\widehat\alpha}^-K^1=K^1U_{\widehat\alpha}^+\tau_{\alpha}$ and $K^1\tau_{\alpha}^{-1}K^1=\tau_{\alpha}^{-1}U_{\widehat\alpha}^+K^1=K^1U_{\widehat\alpha}^-\tau_{\alpha}^{-1}$.
\end{enumerate}
\end{lemma}

\begin{proof}
Point $(a)$ follows from lemma \ref{rmqtau1} since $\tau^{-1} U_{P(\tau)}^+ \tau\subset K^1$ and $\tau U_{P(\tau)}^- \tau^{-1}\subset K^1$.
In order to prove $(b)$ we consider the decomposition $K^1=(K^1\cap U_{P(\tau)}^-)(K^1\cap M_{P(\tau)})(K^1\cap U_{P(\tau)}^+)$. 
Lemma \ref{rmqtau1} implies $\tau^{-1} (K^1\cap U_{P(\tau)}^+)\tau\subset \tau^{-1} U_{P(\tau)}^+\tau \subset K^1$ and 
remark \ref{rmqnormalizzatore} implies that $\tau$ normalizes $K^1\cap M_{P(\tau)}$.
This give rise to $\tau^{-1}K^1\tau K^1= \tau^{-1}(K^1\cap U_{P(\tau)}^-)\tau K^1$.
Furthermore, by lemma \ref{rmqtau1} we have $U_{P(\tau)}^-\subset \tau^{-1}(K^1\cap U_{P(\tau)}^-)\tau$ and 
we also have $\tau^{-1}(K^1\cap U_{P(\tau)}^-)\tau\cap K\subset \tau^{-1} U_{P(\tau)}^-\tau\cap K\subset U_{P(\tau)}^-$. 
This implies $\tau^{-1}(K^1\cap U_{P(\tau)}^-)\tau\cap K=U_{P(\tau)}^-$ and so $\tau^{-1}K^1\tau K^1\cap K=U_{P(\tau)}^- K^1$.
Similarly we can obtain $\tau K^1\tau^{-1} K^1\cap K= U_{P(\tau)}^+ K^1$.
Finally to prove $(c)$ we observe that $\tau_{\alpha}^{-1} K^1 \tau_\alpha$ and $\tau_\alpha K^1 \tau_\alpha^{-1}$ are contained in $K$. 
If we apply $(b)$ we obtain 
$K^1\tau_\alpha K^1= \tau_\alpha U_{\widehat{\alpha}}^-K^1$ and $K^1\tau_\alpha^{-1} K^1=\tau_\alpha^{-1} U_{\widehat{\alpha}}^+K^1$. 
Taking inverses we obtain $K^1\tau_{\alpha}^{-1}K^1=K^1 U_{\widehat\alpha}^-\tau_\alpha^{-1}$ and
$K^1\tau_{\alpha}K^1= K^1 U_{\widehat\alpha}^+\tau_\alpha$.
\end{proof}

\begin{lemma}\label{lemmatau} 
The elements $f_{\tau_{\alpha}}$ with $\alpha\in\Sigma$ commute with each other and if $\tau=\prod_{\alpha\in\Sigma}\tau_{\alpha}^{i_{\alpha}}\in\bm\Delta$ then $f_{\tau}=\prod_{\alpha\in\Sigma}f_{\tau_{\alpha}}^{i_{\alpha}}$. 
\end{lemma}

\begin{proof}
We proceed by induction on the natural number $I(\tau)=\sum_{\alpha\in\Sigma} i_{\alpha}$. 
If $I(\tau)=0$ then $\tau=1$ and the result holds. 
If $I(\tau)>0$ let $\alpha'\in \Sigma$ be such that $i_{\alpha'}>0$. 
We set $\widetilde{\tau}=\tau_{\alpha'}^{-1}\tau\in\bm\Delta$ and we observe that $I(\widetilde{\tau})<I(\tau)$. 
The support of $f_{\tau_{\alpha'}}f_{\widetilde{\tau}}$ is $K^1\tau_{\alpha'} K^1\widetilde{\tau} K^1$. 
By lemma \ref{relgruppi}$(c)$ this support is $K^1U_{\widehat{\alpha'}}^+\tau_{\alpha'}\widetilde{\tau} K^1$ and by lemma \ref{relgruppi}$(a)$ it is $K^1\tau K^1$ since $P(\tau)\subset P(\tau_{\alpha'})$. 
By formula (\ref{eqprodotto2}) and by lemma \ref{relgruppi}$(c)$ we obtain
$f_{\tau_{\alpha'}}f_{\widetilde{\tau}}(\tau)=|(K^1\tau_{\alpha'} K^1\cap \tau K^1\widetilde{\tau}^{\,-1}K^1)/K^1|
=|(\tau_{\alpha'}^{-1}K^1\tau_{\alpha'} K^1\cap \widetilde{\tau} K^1\widetilde{\tau}^{\,-1}K^1)/K^1|
=|(U_{\widehat{\alpha'}}^- K^1\cap \widetilde{\tau}K^1\widetilde{\tau}^{\,-1}K^1)/K^1|$.
By lemma \ref{relgruppi}$(b)$ we have
$U_{\widehat{\alpha'}}^- K^1\cap \widetilde{\tau}K^1\widetilde{\tau}^{\,-1}K^1=U_{\widehat{\alpha'}}^- K^1\cap K\cap \widetilde{\tau}K^1\widetilde{\tau}^{\,-1}K^1=U_{\widehat{\alpha'}}^- K^1 \cap U_{P(\widetilde{\tau})}^+K^1=K^1$
and so $f_{\tau_{\alpha'}}f_{\widetilde{\tau}}(\tau)=1$ which implies $f_{\tau}=f_{\tau_{\alpha'}}f_{\widetilde{\tau}}$.
Commutativity of the $f_{\tau_{\alpha}}$ follows since 
$f_{\tau_{\alpha_1}}f_{\tau_{\alpha_2}}=f_{\tau_{\alpha_1}\tau_{\alpha_2}}=f_{\tau_{\alpha_2}\tau_{\alpha_1}}=f_{\tau_{\alpha_2}}f_{\tau_{\alpha_1}}$ for every $\alpha_1,\alpha_2\in\Sigma$. 
Finally, by induction hypothesis we obtain $f_{\tau}=\prod f_{\tau_{\alpha}}^{i_{\alpha}}$.
\end{proof}

\begin{lemma}\label{scomposizione}
For every element $x$ in $G$ there exist $k_1,k_2\in K^1$, $u_1,u_2\in U$ and unique $t\in T$,  $i\in\Z$, $\tau\in\bm\Delta$, $w_1\in W$ and $w_2$ of minimal length in $W_{P(\tau)}w_2\in W_{P(\tau)}\bs W$ such that $x=k_1u_1t\tau_0^{i}w_1\tau w_2u_2k_2$. 
In particular we have $G=K^1UT\tau_0^{\Z}W\bm\Delta WUK^1$.
\end{lemma}

\begin{proof}  
We know the (Bruhat-Iwahori) decomposition 
\[G=\bigsqcup_{\tilde{w}\in\widetilde{W}} I^1\tilde{w}I^1=\bigsqcup_{\tilde{w}\in\widetilde{W}}K^1U\tilde{w}UK^1\] where $\widetilde{W}$ denotes the group of monomial matrices with coefficients in $\frack_{D}^{\times}\ltimes\varpi^{\Z}$. 
Let us fix $\tilde{w}\in\widetilde{W}$. 
Then there exist $t\in T$, $w\in W$ and $a_j\in\Z$ for every $j\in\{1,\dots,m\}$ such that 
$\tilde{w}=tw \mathrm{diag}(\varpi^{a_1},\dots,\varpi^{a_{m}})$. 
Furthermore, there exists $w_2\in W$ such that 
$\mathrm{diag}(\varpi^{a_1},\dots,\varpi^{a_{m}})=w_2^{-1}\mathrm{diag}(\varpi^{a_{w_2(1)}},\dots,\varpi^{a_{w_2(m)}})w_2$
with $a_{w_2(1)}\leq \cdots\leq a_{w_2(m)}$ and we can choose $w_2$ of minimal length in $W_Pw_2\in W_P\bs W$ where 
$P=\{\alpha_{j,j+1}\in\Sigma\,|\,a_{w_2(j)}=a_{w_2(j+1)}\}$.
Thus we obtain \[\tilde{w}=tww_2^{-1}\tau_0^{a_{w_2(1)}}\prod_{h=1}^{m-1}\tau_h^{a_{w_2(h+1)}-a_{w_2(h)}}w_2.\] 
If we set $w_1=ww_2^{-1}$, $i=a_{w_2(1)}\in\Z$ and $\tau=\prod\tau_h^{a_{w_2(h+1)}-a_{w_2(h)}}\in\bm\Delta$ we obtain
$P=P(\tau)$ and $\tilde{w}=t\tau_0^{i}w_1\tau w_2$. Then we have proved the existence of such a decomposition. 
We now suppose $\tilde{w}=t\tau_0^{i}w_1\tau w_2=t'\tau_0^{i'}w'_1\tau' w'_2$ with $t,t'\in T$, $i,i'\in \Z$, $\tau,\tau'\in\bm\Delta$, 
$w_1,w'_1\in W$, $w_2$ of minimal length in $W_{P(\tau)}w_2$ and $w'_2$ of minimal length in $W_{P(\tau')}w'_2$. 
We have 
\[w_1w_2 (w_2^{-1}w_1^{-1}t\tau_0^{i}w_1w_2w_2^{-1}\tau w_2 )=w'_1w'_2 (w'^{-1}_2w'^{-1}_1t'\tau_0^{i'}w'_1w'_2w'^{-1}_2\tau' w'_2 )\]
and since the two matrices between brackets are diagonal, we obtain $w_1w_2=w'_1w'_2$ and so
\[(w_2^{-1}w_1^{-1}tw_1w_2)(\tau_0^{i}w_2^{-1}\tau w_2 )=(w^{-1}_2w^{-1}_1t'w_1w_2)(\tau_0^{i'}w'^{-1}_2\tau' w'_2).\]
Matrices $w_2^{-1}w_1^{-1}tw_1w_2$ and $w^{-1}_2w^{-1}_1t'w_1w_2$ are in $T$ while matrices $\tau_0^{i}w_2^{-1}\tau w_2$ and $\tau_0^{i'}w'^{-1}_2\tau' w'_2$ have coefficients in $\varpi^{\Z}$.
Hence we obtain $t=t'$ and $\tau_0^{i}w_2^{-1}\tau w_2=\tau_0^{i'}w'^{-1}_2\tau' w'_2$ which implies $\tau_0^{i-i'}\tau=w_2w'^{-1}_2\tau' w'_2w_2^{-1}$.
Since $\tau_0^{i-i'}\tau$ and $\tau'$ are matrices of the form $\mathrm{diag}(\varpi^{b_1},\dots,\varpi^{b_m})$ with $b_1\leq\cdots \leq b_m$, we obtain $w'_2w_2^{-1}\in W_{P(\tau')}$ and so $\tau_0^{i-i'}\tau=\tau'$. 
This implies $i=i'$, $\tau=\tau'$ and so $W_{P(\tau)}w'_2 = W_{P(\tau)}w_2$. 
Finally, since we have chosen $w'_2$ and $w_2$ of minimal length, we obtain $w'_2=w_2$ by lemma \ref{propunicradice}.
\end{proof}

\begin{rmk}\label{rmkscomposizione}
Lemma \ref{scomposizione} also shows that $K=K^1UTWUK^1$ and that for every $k\in K$ there exist $k_1,k_2\in K^1$, $u_1,u_2\in U$ and unique $t\in T$ and $w\in W$ such that $k=k_1u_1twu_2k_2$. 
\end{rmk}

\begin{prop}\label{propgeneratori}
The algebra $\mathscr{H}( G,K^1)$ is generated by $\bm\Omega$ and the subalgebra $\mathscr{H}(K,K^1)$ is generated by 
$f_u,f_t,f_{s_{\alpha}}$ with $u\in U$, $t\in T$ and $\alpha\in \Sigma$.
\end{prop}

\begin{proof}
By lemma \ref{scomposizione} for every $x\in  G$ there exists a decomposition $x=k_1u_1t\tau_0^{i}w_1\tau w_2u_2k_2$ with 
$k_1,k_2\in K^1$, $u_1,u_2\in U$, $t\in T$,  $i\in\Z$, $\tau=\prod\tau_{\alpha}^{i_{\alpha}}\in\bm\Delta$ and $w_1,w_2\in W$.
By lemma \ref{lemmaprodotto} and lemma \ref{lemmatau} we have
\begin{equation}\label{prodottogeneratori}
f_x=f_{u_1t\tau_0^{i}w_1}f_{\tau}f_{w_2u_2}=f_{u_1}f_{t}f_{\tau_0}^if_{w_1}\prod_{\alpha\in\Sigma} f_{\tau_{\alpha}}^{i_{\alpha}}f_{w_2}f_{u_2}
\end{equation}
which proves that $\bm\Omega$ generates $\mathscr{H}(G,K^1)$.
Second assertion follows form decomposition $K=K^1UTWUK^1$ and lemma \ref{lemmaprodotto}.
\end{proof}

\subsection{Relations}\label{sezionerelazioni}
In this paragraph we look for some relations among generators of $\mathscr{H}(G,K^1)$ that we have defined in previous paragraph. 

\begin{prop}\label{proptaueu}
Let $\tau\in\bm\Delta$, $u_1\in U_{P(\tau)}^+$ and $u_2\in U_{P(\tau)}^-$. 
Then $f_{u_1} f_{\tau}=f_{\tau}=f_{\tau}f_{u_2}$.
\end{prop}

\begin{proof}
Elements $u_1$ and $u_2$ normalize $K^1$ and so by lemma \ref{relgruppi}$(a)$ the support of $f_{u_1}f_{\tau}$ and of $f_{\tau} f_{u_2}$ is $K^1\tau K^1$. 
We obtain
$(f_{u_1}f_{\tau})(\tau)=\left|(u_1K^1\cap  \tau K^1\tau^{-1}K^1)/K^1 \right|=1$ and
$(f_{\tau}f_{u_2})(\tau)=\left|(K^1\tau K^1\cap  \tau u_2^{-1} K^1)/K^1 \right|=1$ and hence the result.
\end{proof}

For every $\alpha=\alpha_{i,i+1}\in \Sigma$ and $w\in W$ we consider the following sets:
\begin{enumerate}[$\bullet$]
\item $A(w,\alpha)=\{w(j)\,|\,i+1\leq j\leq m\}\subset \{1,\dots,m\}$,
\item $B(w,\alpha)=\{w(j)-1\,|\,i+1\leq j\leq m\}\subset \{0,\dots,m-1\}$,
\item $P'(w,\alpha)=A(w,\alpha)\setminus B(w,\alpha)$,
\item $P(w,\alpha)=\{\alpha_{i,i+1}\in\Sigma\,|\, i\in P'(w,\alpha)\}$,
\item $Q(w,\alpha)=B(w,\alpha)\setminus A(w,\alpha)$.
\end{enumerate}
We remark that $\tau_{P'(w,\alpha)}=\tau_{P(w,\alpha)}$ because $0\notin P'(w,\alpha)$ and $\tau_m=\mathbb{I}_{m}$.

\smallskip
Let $\alpha=\alpha_{i,i+1}\in \Sigma$, $w'\in W$ and $w$ of minimal length in $w'W_{\widehat {\alpha}}\in W/W_{\widehat {\alpha}}$. 
Then we have 
\[w'\tau_i w'^{-1}=w\tau_iw^{-1}=\prod_{h=i+1}^{m}w\tau_{h-1}\tau_{h}^{-1}w^{-1}=\prod_{h=i+1}^{m}\tau_{w(h)-1}\tau_{w(h)}^{-1}.\]
We obtain  $w\tau_\alpha w^{-1}=\tau_{P'(w,\alpha)}^{-1}\tau_{Q(w,\alpha)}=\tau_{P(w,\alpha)}^{-1}\tau_{Q(w,\alpha)}$. 
This equality suggest us to study the following product in $\mathscr{H}(G,K^1)$:
\begin{equation}\label{rel10}
f_{\tau_{P(w,\alpha)}}f_wf_{\tau_\alpha}f_{w^{-1}}
\end{equation}
and its relation with $f_{\tau_{Q(w,\alpha)}}$.

\begin{lemma}\label{lemmaprodottotau0}
Let $P\subset\Sigma$.  
If $w$ is of minimal length in $W_Pw\in W_P\bs W$ 
then $P\cap P(w,\alpha)=\emptyset$ for every $\alpha\in \Sigma$.
\end{lemma}

\begin{proof}
We fix $\alpha=\alpha_{i,i+1}\in \Sigma$. 
The element $w^{-1}$ is of minimal length in $w^{-1}W_P\in W/W_P$ and then by proposition \ref{proplongmin} we obtain $w^{-1}\bm\Phi_P^+\subset\bm\Phi^+$. 
Moreover for every $\alpha_{k,k+1}\in P(w,\alpha)$ we have $w^{-1}(k)\geq i+1>i\geq w^{-1}(k+1)$ and then
$w^{-1}\alpha_{k,k+1}\in \bm\Phi^-$.  
This implies $P\cap P(w,\alpha)=\emptyset$.
\end{proof}

\begin{rmk}
More generally, we are interested in studying the product $f_\tau f_{w'} f_{\tau_\alpha}$ for every $\tau\in\bm\Delta$, $w'\in W$ and $\alpha\in\Sigma$. 
By remark \ref{rmqnormalizzatore} we have  $f_\tau f_{w'} f_{\tau_\alpha}=f_{w_1}f_\tau f_{w} f_{\tau_\alpha}f_{w_2}$ with 
$w_1\in W_{P(\tau)}$, $w_2\in W_{\widehat{\alpha}}$ and 
$w$ of minimal length in $W_{P(\tau)}w$ and in $wW_{\widehat{\alpha}}$ such that $w'=w_1ww_2$.
Lemma \ref{lemmaprodottotau0} implies that $P(\tau)\cap P(w,\alpha)=\emptyset$ and so there exists $\tau'\in\bm\Delta$ such that
$f_\tau=f_{\tau'}f_{\tau_{P(w,\alpha)}}$.
Thus we obtain $f_\tau f_{w'} f_{\tau_\alpha}=f_{w_1}f_{\tau'}(f_{\tau_{P(w,\alpha)}} f_{w} f_{\tau_\alpha}f_{w^{-1}})f_{ww_2}$ and so we have reduced the study of $f_\tau f_{w'} f_{\tau_\alpha}$ to the study of (\ref{rel10}). 
\end{rmk}

Now we state and prove two lemmas that allow us to compute the support of (\ref{rel10}).

\begin{lemma}\label{lemmaprodottotau1}
Let $\alpha\in\Sigma$, $w\in W$ and $P\subset\Sigma$. 
Then \[K^1\tau_P K^1 w\tau_{\alpha}w^{-1}K^1=K^1\tau_Pw\tau_{\alpha}w^{-1}(U_{\widehat P}^+\cap wU_{\widehat\alpha}^-w^{-1})K^1.\]
\end{lemma}

\begin{proof} By lemma \ref{relgruppi}$(c)$ we have $K^1\tau_{\alpha}K^1=\tau_{\alpha}U_{\widehat\alpha}^-K^1$ and so we obtain 
$K^1\tau_P K^1 w\tau_{\alpha}w^{-1}K^1=K^1\tau_P wK^1\tau_{\alpha}K^1w^{-1}=K^1\tau_P w\tau_{\alpha}U_{\widehat\alpha}^-K^1w^{-1}
=K^1\tau_P w\tau_{\alpha}w^{-1}\Big(\prod_{\alpha'\in w\mathbf{\Psi}_{\widehat\alpha}^-}U_{\alpha'}\Big)K^1$.
If $\alpha'\in w\mathbf{\Psi}_{\widehat\alpha}^-$ then by lemma \ref{rmqtau1} we have
$w\tau_{\alpha}w^{-1} U_{\alpha'} (w\tau_{\alpha}w^{-1})^{-1}=w(\tau_{\alpha} U_{w^{-1}\alpha'}\tau_{\alpha}^{-1})w^{-1}=
U_{\alpha'}\cap K^1$
and so by remark \ref{rmqnormalizzatore} and lemma \ref{rmqtau1} if 
$\alpha'$ is in $w\mathbf{\Psi}_{\widehat\alpha}^-$ and not in $\mathbf{\Psi}_{\widehat P}^+$
then 
$(\tau_P w\tau_{\alpha}w^{-1}) U_{\alpha'}(\tau_P w\tau_{\alpha}w^{-1})^{-1}$ is contained in $K^1$.
Thus we obtain 
$K^1\tau_P K^1 w\tau_{\alpha}w^{-1}K^1=K^1\tau_P w\tau_{\alpha}w^{-1}\Big(\prod_{\alpha'\in w\mathbf{\Psi}_{\widehat\alpha}^-\cap 
\mathbf{\Psi}_{\widehat P}^+}U_{\alpha'}\Big)K^1$
and hence the result.
\end{proof}

\begin{lemma}\label{lemmaprodottotau2}
Let $\alpha=\alpha_{i,i+1}\in\Sigma$, $w\in W$ and $P=P(w,\alpha)$. 
\begin{enumerate}[$(a)$]
\item The sets 
$\mathbf{\Phi}_{\widehat P}^+\cap w \mathbf{\Psi}_{\widehat\alpha}^-$ and $\mathbf{\Phi}_{\widehat P}^-\cap w \mathbf{\Psi}_{\widehat\alpha}^+$
are empty.
\item We have $U_{\widehat P}^+\cap wU_{\widehat\alpha}^-w^{-1}=U\cap wU_{\widehat\alpha}^-w^{-1}$
and $U_{\widehat P}^-\cap wU_{\widehat\alpha}^+w^{-1}=U^-\cap wU_{\widehat\alpha}^+w^{-1}$.
\item If in addition $w$ is of minimal length in $w W_{\widehat\alpha}\in W/ W_{\widehat\alpha}$ then 
$U\cap wU_{\widehat\alpha}^-w^{-1}=U\cap wU^-w^{-1}$ and
$U^-\cap wU_{\widehat\alpha}^+w^{-1}=U^-\cap wUw^{-1}$.
\end{enumerate}
\end{lemma}

\begin{proof} 
In order to prove $(a)$ we suppose by contradiction that there exists $\alpha'\in\mathbf{\Phi}_{\widehat P}^+\cap w \mathbf{\Psi}_{\widehat\alpha}^-$. 
Since $\alpha'\in w \mathbf{\Psi}_{\widehat\alpha}^-$, we have $\alpha'=\alpha_{w(h)w(k)}$ for some $i+1\leq h\leq m$ and $1\leq k\leq i$ and 
since $\alpha'\in \mathbf{\Phi}_{\widehat P}^+$ we have $\alpha_{w(h)w(k)}=\alpha_{w(h),w(h)+1}+\cdots+\alpha_{w(k)-1,w(k)}$ with 
$\alpha_{w(h),w(h)+1},\cdots,\alpha_{w(k)-1,w(k)}\notin P(w,\alpha)$ and so $w(h),\dots,w(k)-1\notin P'(w,\alpha)$. 
By construction, we have $w(h)\in A(w,\alpha)$ which implies $w(h)\in B(w,\alpha)$ and $w(h)+1\in A(w,\alpha)$. 
By induction we obtain $w(k)\in A(w,\alpha)$ which contradicts $1\leq k\leq i$. 
So we have proved that $\mathbf{\Phi}_{\widehat P}^+\cap w \mathbf{\Psi}_{\widehat\alpha}^-$ is empty. 
Furthermore, we remark that $\mathbf{\Phi}_{\widehat P}^-\cap w \mathbf{\Psi}_{\widehat\alpha}^+=-(\mathbf{\Phi}_{\widehat P}^+\cap w 
\mathbf{\Psi}_{\widehat\alpha}^-)$.
To prove $(b)$ we observe that by $(a)$ we have  
$\mathbf{\Psi}_{\widehat P}^+\cap w \mathbf{\Psi}_{\widehat\alpha}^-=\mathbf{\Phi}^+\cap w \mathbf{\Psi}_{\widehat\alpha}^-$ and
$\mathbf{\Psi}_{\widehat P}^-\cap w \mathbf{\Psi}_{\widehat\alpha}^+=\mathbf{\Phi}^-\cap w \mathbf{\Psi}_{\widehat\alpha}^+$. 
Thus we obtain
 $U_{\widehat P}^+\cap wU_{\widehat\alpha}^-w^{-1}=U\cap wU_{\widehat\alpha}^-w^{-1}$ and 
 $U_{\widehat P}^-\cap wU_{\widehat\alpha}^+w^{-1}=U^-\cap wU_{\widehat\alpha}^+w^{-1}$.
Finally to prove $(c)$ we remark that by proposition \ref{proplongmin} we have $w\mathbf{\Phi}_{\widehat\alpha}^-\subset \bm\Phi^-$ and 
$w\mathbf{\Phi}_{\widehat\alpha}^+\subset \bm\Phi^+$. 
We obtain 
$\bm\Phi^+\cap w\bm\Psi_{\widehat\alpha}^-=\bm\Phi^+\cap w\bm\Phi^-$ and 
$\bm\Phi^-\cap w\bm\Psi_{\widehat\alpha}^+=\bm\Phi^-\cap w\bm\Phi^+$
 and hence the result.
\end{proof}

We now have all the tools necessary to find the relation between (\ref{rel10}) and $f_{\tau_{Q(w,\alpha)}}$.

\begin{prop}\label{proprel9}
Let $\alpha\in\Sigma$ and $w$ be of minimal length in $wW_{\widehat{\alpha}}\in W/W_{\widehat{\alpha}}$. 
Then we have
\[f_{\tau_{P(w,\alpha)}}f_wf_{\tau_\alpha}f_{w^{-1}}=q^{\ell(w)}f_{\tau_{Q(w,\alpha)}}\sum_{u}f_u.\]
where $u$ describes a system of representatives of $(U\cap wU^-	w^{-1})K^1/K^1$ in $U\cap wU^-	w^{-1}$.
\end{prop}

\begin{proof}
We write $P=P(w,\alpha)$ and $Q=Q(w,\alpha)$. 
By lemmas \ref{lemmaprodottotau1} and \ref{lemmaprodottotau2} the support of 
$f_{\tau_{P}}f_wf_{\tau_\alpha}f_{w^{-1}}$ is $K^1\tau_{P}w\tau_{\alpha}w^{-1}(U_{\widehat P}^+\cap wU_{\widehat\alpha}^-w^{-1})K^1=K^1\tau_Q(U\cap wU^-w^{-1})K^1$.
Now if we take $u\in U\cap wU^-w^{-1}=U_{\widehat P}^+\cap wU_{\widehat\alpha}^-w^{-1}$ we obtain
\begin{align*}
(f_{\tau_{P}}f_{w\tau_\alpha w^{-1}})(\tau_Qu)&=\left|(K^1\tau_PK^1\cap \tau_Q uK^1w\tau_{\alpha}^{-1}w^{-1}K^1)/K^1\right|\\
&=\left|(\tau_P^{-1}K^1\tau_PK^1\cap w\tau_{\alpha}w^{-1}uK^1w\tau_{\alpha}^{-1}w^{-1}K^1)/K^1\right|\\
&=\left|(\tau_P^{-1}K^1\tau_PK^1\cap w\tau_{\alpha}K^1(w^{-1}uw)\tau_{\alpha}^{-1}K^1w^{-1})/K^1\right|\\
\text{\scriptsize{(lemma \ref{rmqtau1})}}
&=\left|(\tau_P^{-1}K^1\tau_PK^1\cap w\tau_{\alpha}K^1\tau_{\alpha}^{-1}K^1w^{-1})/K^1\right|\\
\text{\scriptsize{(lemma \ref{relgruppi}$(c)$)}}
&=\left|(\tau_P^{-1}K^1\tau_PK^1 \cap wU_{\widehat\alpha}^+w^{-1}K^1)/K^1\right|.
\end{align*}
To simplify notations, we denote temporarily $V=wU_{\widehat\alpha}^+w^{-1}$. 
Since $VK^1\subset K$, by lemma \ref{relgruppi}$(b)$ we have 
$\tau_P^{-1}K^1\tau_PK^1 \cap VK^1=\tau_P^{-1}K^1\tau_PK^1\cap K \cap VK^1 = U_{\widehat{P}}^-K^1 \cap VK^1$.
We want to prove that this last intersection is equal to $(U^-\cap V)K^1$. 
We consider the decomposition $V=(U^-\cap V)(U\cap V)$.
Let $x=v_1v_2k\in U_{\widehat{P}}^-K^1 \cap VK^1$ with $v_1\in U^-\cap V$, $v_2\in U\cap V$ and $k\in K^1$.
We have $v_1^{-1}x \in (U^-\cap V)U_{\widehat{P}}^-K^1 \cap (U\cap V)K^1\subset U^-K^1\cap UK^1=K^1$
and then $U_{\widehat{P}}^-K^1 \cap VK^1\subset (U^-\cap V)K^1$.
By lemma \ref{lemmaprodottotau2}$(a)$ we have $\mathbf{\Phi}_{\widehat P}^-\cap w \mathbf{\Psi}_{\widehat\alpha}^+=\emptyset$ which implies  
$U_{\widehat{P}}^- \cap V=U^-\cap V$ and so
$(U^-\cap V)K^1=(U_{\widehat{P}}^-\cap V)K^1\subset U_{\widehat{P}}^-K^1 \cap VK^1$.\\
This implies $(f_{\tau_{P}}f_{w\tau_\alpha w^{-1}})(\tau_Q u)=\left|(U^-\cap wU_{\widehat\alpha}^+w^{-1})K^1/K^1\right|$.
Since $w$ is of minimal length in $wW_{\widehat\alpha}$, by lemma  \ref{lemmaprodottotau2}$(c)$ we have  
$U^-\cap wU_{\widehat\alpha}^+w^{-1}=U^-\cap wUw^{-1}$ and then
$(f_{\tau_{P}}f_{w\tau_\alpha w^{-1}})(\tau_Q u)=\left|(U^-\cap wUw^{-1})K^1/K^1\right|$
which is $q^{\ell(w)}$ by remark \ref{rmqlunghezza}.
Finally we remark that $\sum_u f_u$ where $u$ describes a system of representatives of $(U\cap wU^-	w^{-1})K^1/K^1$ in $U\cap wU^-	w^{-1}$ is the characteristic function of $U\cap wU^- w^{-1}$.
This implies $(f_{\tau_{P}}f_wf_{\tau_\alpha}f_{w^{-1}})(\tau_Q u)=q^{\ell(w)}(f_{\tau_{Q(w,\alpha)}}\sum_{u}f_u) (\tau_Q u)$ for every $u\in U\cap wU^-w^{-1}$ and hence the result.
\end{proof}

\subsection{Presentation by generators and relations}
Now we want to prove that the list of relations that we have found in previous paragraph allows us to obtain a presentation of the algebra $\mathscr{H}(G,K^1)$. 

\begin{defin}\label{defA}
Let $\mathscr{A}$ be the $\Z$-algebra generated by elements $\f_\omega$ with $\omega\in \Omega$ which satisfy the following relations. 
\begin{enumerate}[1.]
	\item  $\f_{k}=1$ for every $k\in K^1$ and $\f_{k_1}\f_{k_2}=\f_{k_1k_2}$ for every $k_1,k_2\in K$;
	\item $\f_{\tau_0}\f_{\tau_0^{-1}}=1$
				and $\f_{\tau_0^{-1}}\f_{\omega}=\f_{\tau_0^{-1}\omega\tau_0}\f_{\tau_0^{-1}}$ for every $\omega\in \Omega$;
	\item $\f_{\tau_{\alpha}}\f_t=\f_{\tau_{\alpha}t\tau_{\alpha}^{-1}}\f_{\tau_{\alpha}}$ for every $t\in T$ and $\alpha\in\Sigma$;
  \item $\f_{\tau_{\alpha}}\f_u=\f_{\tau_{\alpha}u\tau_{\alpha}^{-1}}\f_{\tau_{\alpha}}$ if $u\in U_{\alpha'}$ with $\alpha'\in \mathbf{\Phi}_{\widehat{\alpha}}$, for every $\alpha\in\Sigma$;
  \item $\f_u\f_{\tau_{\alpha}}=\f_{\tau_{\alpha}}$ if $u\in U_{\alpha'}$ with $\alpha'\in \mathbf{\Psi}_{\widehat{\alpha}}^+$, for every $\alpha\in\Sigma$;
	\item $\f_{\tau_{\alpha}}\f_u=\f_{\tau_{\alpha}}$ if $u\in U_{\alpha'}$ with $\alpha'\in \mathbf{\Psi}_{\widehat{\alpha}}^-$, for every $\alpha\in\Sigma$;
	\item $\f_{\tau_{\alpha}}\f_{s_{\alpha'}}=\f_{s_{\alpha'}}\f_{\tau_{\alpha}}$ for every $\alpha\neq\alpha'$ in $\Sigma$;
	\item $\f_{\tau_{\alpha}}\f_{\tau_{\alpha'}}=\f_{\tau_{\alpha'}}\f_{\tau_{\alpha}}$ for every $\alpha,\alpha'\in\Sigma$;
	\item $\displaystyle{\Bigg(\prod_{\alpha'\in P(w,\alpha)}\f_{\tau_{\alpha'}}\Bigg)\f_w\f_{\tau_\alpha}\f_{w^{-1}}=
	q^{\ell(w)}\Bigg(\prod_{\alpha''\in Q(w,\alpha)}\f_{\tau_{\alpha''}}\Bigg)\left(\sum_{u}\f_u\right)}\;\;$ 
	for every $\alpha\in\Sigma$ and $w$ of minimal length in $wW_{\widehat \alpha}\in W/W_{\widehat \alpha}$ and where $u$ 
	describes a system of representatives of $(U\cap wU^-	w^{-1})K^1/K^1$ in $U\cap wU^-	w^{-1}$. 
\end{enumerate}
\end{defin}

\begin{rmk}\label{rmktau0}
By relation 2 of definition \ref{defA} we have $\f_{\tau_0^{-1}}\f_{\tau_0}=\f_{\tau_0^{-1}\tau_0\tau_0}\f_{\tau_0^{-1}}=\f_{\tau_0}\f_{\tau_0^{-1}}=1$ and so $\f_{\tau_0}$ has a two-sides inverse. 
Moreover we have 
$\f_{\tau_0}\f_{\omega}=\f_{\tau_0}\big(\f_{\tau_0^{-1}(\tau_0\omega\tau_0^{-1})\tau_0}\f_{\tau_0^{-1}}\big)\f_{\tau_0}=\f_{\tau_0}\f_{\tau_0^{-1}}\f_{\tau_0\omega\tau_0^{-1}}\f_{\tau_0}=\f_{\tau_0\omega\tau_0^{-1}}\f_{\tau_0}$ for every $\omega\in\Omega$.
Furthermore, if $i$ is a negative integer we write $\f_{\tau_0}^i=\f_{\tau_0^{-1}}^{-i}\in\mathscr{A}$ and then 
$\f_{\tau_0^{i}}\f_{\omega}=\f_{\tau_0^{i}\omega\tau_0^{-i}}\f_{\tau_0^{i}}$ holds for every $\omega\in \Omega$ and $i\in\Z$.
\end{rmk}

\begin{rmk}
If $D=F$ then $\f_{\tau_0^{-1}}$ and $\f_{\tau_0}$ are in the centre of the algebra. 
Moreover for every $\alpha\in\Sigma$ the element 
$\f_{\tau_{\alpha}}$ commutes with $\f_t$ with $t\in T$ and with $\f_u$ where $u\in U_{\alpha'}$ with $\alpha'\in \mathbf{\Phi}_{\widehat{\alpha}}$.
\end{rmk}

In order to simplify notation, from now on if 
$\tau=\prod_{\alpha\in \Sigma}\tau_\alpha^{i_\alpha}\in\bm\Delta$, where $i_\alpha\in \N$ for every $\alpha\in\Sigma$, we write
$\f_\tau=\prod\f_{\tau_\alpha}^{i_\alpha}\in\mathscr{A}$ that is well defined by 8 of definition \ref{defA}. 

\begin{prop}
The map 
\begin{equation*}
\begin{array}{rccc}
\vartheta:&\mathscr{A}&\longrightarrow & \mathscr{H}(G,K^1)\\
&\f_\omega&\longmapsto& f_\omega
\end{array}
\end{equation*}
is a surjective algebra homomorphism.
\end{prop}

\begin{proof}
We only need to show that the $f_{\omega}$ with $\omega\in\Omega$ satisfy all relations of definition \ref{defA}.
Thanks to lemma \ref{lemmaprodotto} and to remark \ref{rmqnormalizzatore} we obtain relations 1 and 2 because $K$ and $\tau_0$ normalize $K^1$, relation 3 because $\tau_{\alpha}$ normalizes $T$, relation 4 because $\tau_{\alpha}$ normalizes $U_{\alpha'}$ with 
$\alpha'\in \mathbf{\Phi}_{\widehat{\alpha}}$ and relation 7 because $\tau_{\alpha}$ commute with $s_{\alpha'}$ if $\alpha\neq\alpha'$. 
Furthermore, we obtain relations 5 and 6 by proposition \ref{proptaueu}, relation 8 by lemma \ref{lemmatau} and relation 9 by proposition \ref{proprel9}. 
\end{proof}

Now, we want to prove that  $\vartheta$ is an isomorphism  constructing a surjective homomorphism of $\Z$-modules 
$\vartheta':\mathscr{H}(G,K^1)\longrightarrow\mathscr{A}$ such that $\vartheta(\vartheta'(\Phi))=\Phi$ for every $\Phi\in \mathscr{H}( G,K^1)$. 

\begin{lemma}
The map
\begin{equation}\label{inversa}
\begin{array}{ccc}
\bm\Omega&\longrightarrow & \mathscr{A}\\
f_\omega&\longmapsto& \f_\omega
\end{array}
\end{equation}
is well defined.
\end{lemma}

\begin{proof}
Let $\omega_1,\omega_2\in\Omega$ be such that $K^1\omega_1K^1=K^1\omega_2K^1$. 
If $\omega_1\in K$ then $\omega_1^{-1}\omega_2\in K^1$ and by relation 1 of definition \ref{defA} we have 
$\f_{\omega_1}=\f_{\omega_2}$. 
If $\omega_1\in\{\tau_0,\tau_0^{-1}\}\cup \{\tau_\alpha\,|\,\alpha\in\Sigma\}$ then $\omega_2\in K^1\omega_1 K^1\cap \Omega$ and so 
$\omega_2=\omega_1$.
\end{proof}

We want to extend (\ref{inversa}) to a homomorphism of $\Z$-modules $\vartheta':\mathscr{H}(G,K^1)\longrightarrow \mathscr{A}$. 
Since we can write every $\Phi\in \mathscr{H}(G,K^1)$ uniquely as $\Phi=\sum_{x\in\Xi}\alpha_x f_x$ with $\alpha_x\in \Z$, it is sufficient to define the image of $f_x$ for every $x\in \Xi$. 
In order to do it we state and prove some lemmas.

\begin{lemma}\label{coroldefA}
Let $\tau\in\bm\Delta$. 
Then for every $u_1\in U_{P(\tau)}^+$, $u_2\in U_{P(\tau)}^-$ and $u_3\in M_{P(\tau)}^+\cup M_{P(\tau)}^-$ we have  
$\f_{u_1} \f_{\tau}=\f_{\tau}=\f_{\tau}\f_{u_2}$ and $\f_\tau \f_{u_3}=\f_{\tau u_3\tau^{-1}} \f_\tau$.
\end{lemma}

\begin{proof}
We remark that $\tau_{\widehat{P(\tau)}}$ divides $\tau$ in $\bm\Delta$. 
Since 
$U_{P(\tau)}^+=\prod_{\alpha'\in\mathbf{\Psi}_{P(\tau)}^+}U_{\alpha'}$ and $\bm\Psi_{P(\tau)}^+=\bigcup_{\alpha\in \widehat{P(\tau)}}\bm\Psi_{\widehat{\alpha}}^+$, then by 5 and 8 of definition \ref{defA} we have $\f_{u_1} \f_{\tau}=\f_{\tau}$. 
Similarly by 6 and 8 of definition \ref{defA} we obtain $ \f_{\tau} \f_{u_2}=\f_{\tau}$. 
Since $M_{P(\tau)}^+=\prod_{\alpha\in\mathbf{\Phi}_{P(\tau)}^+}U_{\alpha}$ and 
$\bm\Phi_{P(\tau)}^+=\bigcap_{\alpha\in \widehat{P(\tau)}}\bm\Phi_{\widehat\alpha}^+$, then
by 4 and 8 of definition \ref{defA} we obtain $\f_\tau \f_{u_3}=\f_{\tau u_3\tau^{-1}} \f_\tau$
for every $u_3\in M_{P(\tau)}^+$.
Similarly we have $M_{P(\tau)}^-=\prod_{\alpha\in\mathbf{\Phi}_{P(\tau)}^-}U_{\alpha}$ and 
$\bm\Phi_{P(\tau)}^-=\bigcap_{\alpha\in \widehat{P(\tau)}}\bm\Phi_{\widehat\alpha}^-$ and so 
$\f_\tau \f_{u_3}=\f_{\tau u_3\tau^{-1}} \f_\tau$
for every $u_3\in M_{P(\tau)}^-$.
\end{proof}

Let $x\in G$. 
By (\ref{prodottogeneratori}) there exist $u_1,u_2\in U$, $t\in T$,  $i\in\Z$, $\tau=\prod_{\alpha\in\Sigma} \tau_\alpha^{i_\alpha}\in\bm\Delta$ and $w_1,w_2\in W$ such that 
$f_x=f_{u_1}f_{t}f_{\tau_0}^if_{w_1}f_\tau f_{w_2}f_{u_2}$.
We want to prove that the map $f_x\longmapsto \f_{u_1}\f_{t}\f_{\tau_0}^i\f_{w_1}\f_\tau \f_{w_2}\f_{u_2}$ is well defined.

\begin{lemma}\label{lemmadecompminim}
Let $t\in T$,  $i\in\Z$, $\tau\in\bm\Delta$, $w_1\in W$ and $w_2$ be of minimal length in $W_{P(\tau)}w_2$.  
Then for every $u_1,u_2\in U$ there exist $\widetilde{u_1}\in w_1 U^-w_1^{-1}\cap U$ and $\widetilde{u_2}\in w_2^{-1} U w_2\cap U$ such that
$\f_{u_1}\f_{t}\f_{\tau_0}^i\f_{w_1}\f_\tau \f_{w_2}\f_{u_2}=\f_{\widetilde{u_1}}\f_{t}\f_{\tau_0}^i\f_{w_1}\f_\tau \f_{w_2}\f_{\widetilde{u_2}}$.
\end{lemma}

\begin{proof}
By relation 1 of definition \ref{defA} and by remark \ref{rmktau0} we obtain $\f_{u_1}\f_{t}\f_{\tau_0}^i\f_{w_1}\f_{\tau}\f_{w_2}\f_{u_2}=\f_t\f_{\tau_0}^i\f_{u'_1}\f_{w_1}\f_{\tau}\f_{w_2}\f_{u_2}$ where $u'_1=\tau_0^{-i}t^{-1}u_1t\tau_0^i\in U$. 
We set $P=P(\tau)$ and we consider the decomposition $U=(w_2^{-1}U_P^-w_2\cap U)(w_2^{-1}M_P^-w_2\cap U)(w_2^{-1}Uw_2\cap U)$.
Since $w_2$ is of minimal length in $W_Pw_2$, then $w_2^{-1}$ is of minimal length in $w_2^{-1}W_P$ and so by proposition \ref{proplongmin} we have $w_2^{-1}\mathbf{\Phi}_P^-\subset \mathbf{\Phi}^-$. 
We obtain $w_2^{-1}M_P^-w_2\cap U=\{1\}$ and then $U=(w_2^{-1}U_P^-w_2\cap U)(w_2^{-1}Uw_2\cap U)$.
Let $u_2=u_{2,1}u_{2,2}$ be according to this decomposition. 
By relation 1 of definition \ref{defA} and by lemma \ref{coroldefA} we have
$\f_{\tau}\f_{w_2}\f_{u_2}=\f_{\tau}\f_{w_2u_{2,1}w_2^{-1}}\f_{w_2}\f_{u_{2,2}}=\f_{\tau}\f_{w_2}\f_{u_{2,2}}$.
We now consider the decomposition $U=(w_1U^-w_1^{-1}\cap U)(w_1M_P^+w_1^{-1}\cap U)(w_1U_P^+w_1^{-1}\cap U)$
and we take  $u'_1=u_{1,1}u_{1,2}u_{1,3}$ according to this decomposition. 
By 1 of definition \ref{defA} and by lemma \ref{coroldefA} we have
$\f_{u'_1}\f_{w_1}\f_{\tau}
=\f_{u_{1,1}}\f_{w_1}\f_{w_1^{-1}u_{1,2}w_1}\f_{w_1^{-1}u_{1,3}w_1}\f_{\tau}
=\f_{u_{1,1}}\f_{w_1}\f_{w_1^{-1}u_{1,2}w_1}\f_{\tau}
=\f_{u_{1,1}}\f_{w_1}\f_{\tau}\f_{\tau^{-1}w_1^{-1}u_{1,2}w_1\tau}$.
Thus we obtain
$\f_{u'_1}\f_{w_1}\f_{\tau}\f_{w_2}\f_{u_2}=\f_{u_{1,1}}\f_{w_1}\f_{\tau}\f_{w_2}\f_{w_2^{-1}\tau^{-1}w_1^{-1}u_{1,2}w_1\tau w_2}\f_{u_{2,2}}$.
Now, since by proposition \ref{proplongmin} we have $w_2^{-1}\mathbf{\Phi}_P^+\subset \mathbf{\Phi}^+$, we obtain
$w_2^{-1}(\tau^{-1}w_1^{-1}u_{1,2}w_1\tau) w_2\in w_2^{-1}M_P^+w_2\subset w_2^{-1}Uw_2\cap U$.
So, if we set $\widetilde{u_1}=t\tau_0^iu_{1,1}\tau_0^{-i}t^{-1}\in w_1U^-w_1^{-1}\cap U$ and $\widetilde{u_2}=w_2^{-1}\tau^{-1}w_1^{-1}u_{1,3}w_1\tau w_2u_{2,2}\in w_2^{-1}Uw_2\cap U$ we obtain the result.
\end{proof}

\begin{lemma}\label{lemmadecompminimale}
Let $x\in G$. 
There exist $t\in T$, $i\in\Z$, $\tau\in\bm\Delta$, $w_1\in W$, $w_2$ of minimal length in $W_{P(\tau)}w_2\in W_{P(\tau)}\bs W$, $u_1\in w_1U^-w_1^{-1}\cap U$ and $u_2\in w_2^{-1}Uw_2\cap U$ such that 
\begin{equation}\label{prodottogeneratori2}
f_x=f_{u_1}f_{t}f_{\tau_0}^if_{w_1}f_\tau f_{w_2}f_{u_2}.
\end{equation}
Moreover such a decomposition of $f_x$ is unique. 
\end{lemma}

\begin{proof}
By formula (\ref{prodottogeneratori}) and by lemma \ref{scomposizione}, there exist $u_1,u_2\in U$ and unique $t\in T$,  $i\in\Z$, $\tau\in\bm\Delta$, $w_1\in W$ and $w_2$ of minimal length in $W_{P(\tau)}w_2$ such that $f_x=f_{u_1}f_{t}f_{\tau_0}^if_{w_1}f_\tau f_{w_2}f_{u_2}=\vartheta(\f_{u_1}\f_{t}\f_{\tau_0}^i\f_{w_1}\f_\tau \f_{w_2}\f_{u_2})$.
By lemma \ref{lemmadecompminim} we can choose $u_1\in w_1U^-w_1^{-1}\cap U$ and $u_2\in w_2^{-1}Uw_2\cap U$ and so 
we have proved the existence of such a decomposition.  
Now, let $u_1,u'_1\in w_1U^-w_1^{-1}\cap U$ and $u_2,u'_2\in w_2^{-1}Uw_2\cap U$ be such that 
$f_x=f_{u_1}f_{t}f_{\tau_0}^if_{w_1}f_{\tau}f_{w_2}f_{u_2}=f_{u'_1}f_{t}f_{\tau_0}^if_{w_1}f_{\tau}f_{w_2}f_{u'_2}$. 
By lemma \ref{lemmaprodotto} we obtain
\[f_{\tau}=f_{w_1^{-1}\tau_0^{-i}t^{-1}u_1^{-1}u'_1t\tau_0^iw_1}f_{\tau}f_{w_2u'_2u_2^{-1}w_2^{-1}}=f_{u_-}f_{\tau}f_{u_+}\]
where $u_-=w_1^{-1}\tau_0^{-i}t^{-1}u_1^{-1}u'_1t\tau_0^iw_1\in U^-$ and $u_+=w_2u'_2u_2^{-1}w_2^{-1}\in U$. 
By remark \ref{rmkscomposizione} there exist $\tilde{u}_1\in U$, $\tilde{t}\in T$, $\tilde{w}\in W$, $\tilde{u}_2\in M_P^+$ and $\tilde{u}_3\in U_P^+$ such that $f_{u_-}=f_{\tilde{u}_1}f_{\tilde{t}}f_{\tilde{w}}f_{\tilde{u}_2}f_{\tilde{u}_3}$.
By proposition \ref{proptaueu} and lemma \ref{lemmaprodotto} we obtain 
$f_{\tau}=f_{\tilde{u}_1}f_{\tilde{t}}f_{\tilde{w}}f_{\tilde{u}_2}f_{\tilde{u}_3}f_{\tau}f_{u_+}=f_{\tilde{u}_1}f_{\tilde{t}}f_{\tilde{w}}f_{\tau}f_{\tau^{-1}\tilde{u}_2\tau u_+}$ and then $\tilde{t}=\tilde{w}=1$ by lemma \ref{scomposizione}. 
We obtain $f_{u_-}=f_{\tilde{u}_1\tilde{u}_2\tilde{u}_3}$ which implies $u_-\in U^-\cap K^1U\subset K^1$ and so $f_{\tau}=f_{\tau}f_{u_+}$. 
This implies  $K^1\tau K^1\cap K^1\tau u_+K^1\neq\emptyset$ and so by lemma \ref{relgruppi}$(b)$ we have
$u_+\in \tau^{-1}K^1\tau K^1\cap U\subset U^-K^1\cap U\subset K^1$. 
We have proved $f_{u_-}=f_{u_+}=1$ which implies  $f_{u'_1}=f_{u_1}$ and $f_{u'_2}=f_{u_2}$, hence the uniqueness of decomposition  
(\ref{prodottogeneratori2}).
\end{proof}

Let $x\in\Xi$ and let $f_x=f_{u_1}f_{t}f_{\tau_0}^if_{w_1}f_\tau f_{w_2}f_{u_2}$ be the unique decomposition of $f_x$ of the form (\ref{prodottogeneratori2}). 
We define $\vartheta'(f_x)=\f_{u_1}\f_{t}\f_{\tau_0}^i\f_{w_1}\f_\tau \f_{w_2}\f_{u_2}$ and the homomorphism of $\Z$-modules
\begin{equation}\label{inversa2}
\begin{array}{ccc}
\vartheta':\mathscr{H}(G,K^1)&\longrightarrow &\mathscr{A}\\
\sum_{x\in\Xi}a_x f_x&\longmapsto& \sum_{x\in\Xi}a_x \vartheta'(f_x).
\end{array}
\end{equation}
Now we want to prove that $\vartheta'$ is surjective. 

\begin{lemma}\label{lemmaprodotto2}
Let $x\in G$ and $f_x=f_{u_1}f_{t}f_{\tau_0}^if_{w_1}f_\tau f_{w_2}f_{u_2}$ be a decomposition of the form (\ref{prodottogeneratori}). 
Then $\vartheta'(f_x)=\f_{u_1}\f_{t}\f_{\tau_0}^i\f_{w_1}\f_\tau \f_{w_2}\f_{u_2}$. 
\end{lemma}

\begin{proof}
If $w_2=w_3w_4$ with $w_3\in W_{P(\tau)}$ and $w_4$ is of minimal length in $W_{P(\tau)}w_2$ then by 1 and 7 of definition \ref{defA} we have  
$\f_{u_1}\f_{t}\f_{\tau_0}^i\f_{w_1}\f_\tau \f_{w_2}\f_{u_2}=\f_{u_1}\f_{t}\f_{\tau_0}^i\f_{w_1w_3}\f_\tau \f_{w_4}\f_{u_2}$.
By lemma \ref{lemmadecompminim} there exist $\widetilde{u_1}\in w_1w_3 U^-w_3^{-1}w_1^{-1}\cap U$ and $\widetilde{u_2}\in w_4^{-1} U w_4\cap U$ such that $\f_{u_1}\f_{t}\f_{\tau_0}^i\f_{w_1w_3}\f_\tau \f_{w_4}\f_{u_2}=\f_{\widetilde{u_1}}\f_{t}\f_{\tau_0}^i\f_{w_1w_3}\f_\tau \f_{w_4}\f_{\widetilde{u_2}}$ in $\mathscr{A}$.
Since $f_x=\vartheta(\f_{u_1}\f_{t}\f_{\tau_0}^i\f_{w_1}\f_\tau \f_{w_2}\f_{u_2})=\vartheta(\f_{\widetilde{u_1}}\f_{t}\f_{\tau_0}^i\f_{w_1w_3}\f_\tau \f_{w_4}\f_{\widetilde{u_2}}) =f_{\widetilde{u_1}}f_{t}f_{\tau_0}^if_{w_1w_3}f_\tau f_{w_4}f_{\widetilde{u_2}}$ is the unique decomposition of the form (\ref{prodottogeneratori2}), we obtain that
$\vartheta'(f_x)=
\f_{\widetilde{u_1}}\f_{t}\f_{\tau_0}^i\f_{w_1w_3}\f_\tau \f_{w_4}\f_{\widetilde{u_2}}=
\f_{u_1}\f_{t}\f_{\tau_0}^i\f_{w_1}\f_\tau \f_{w_2}\f_{u_2}$.
\end{proof}

\begin{prop}\label{prodotto}
The homomorphism of $\Z$-modules $\vartheta'$ defined in (\ref{inversa2}) is surjective. 
\end{prop}

\begin{proof}
Let $\mathbf{M}$ be the free multiplicative monoid on the set of $\tilde{\f}_\omega$ with $\omega\in\Omega$,
let $\mathbf{A}$ be the multiplicative submonoid of $\mathscr{A}$ generated by the $\f_\omega$ with $\omega\in\Omega$
and let $\phi:\mathbf{M}\longrightarrow \mathbf{A}$ be the natural projection.  
Since $\vartheta'$ is a homomorphism of $\Z$-modules, it is sufficient to prove that every element $\mathbf{a}\in \mathbf{A}$ (or equivalently  every $\phi(\mathbf{m})$ with $\mathbf{m}\in \mathbf{M}$) is in the image of $\vartheta'$. 
We prove this result by induction on the length of $\mathbf{m}\in \mathbf{M}$. 
We have $\phi(\tilde{\f}_\omega)=\f_\omega=\vartheta'(f_\omega)$ for every $\omega\in\Omega$. 
Now we suppose $\phi(\mathbf{m})=\vartheta'(\Phi_{\mathbf{m}})\in\mathscr{A}$ with $\Phi_{\mathbf{m}}\in\mathscr{H}(G,K^1)$ 
and we prove that $\phi(\mathbf{m}\tilde{\f}_\omega)=\vartheta'(\Phi_{\mathbf{m}})\f_\omega$ is in the image of $\vartheta'$ for every $\omega\in\Omega$.
Thanks to 1 of definition \ref{defA}, to remark \ref{rmkscomposizione} and to fact that $\vartheta'$ is a homomorphism of $\Z$-modules, it is sufficient to prove that for every $x\in G$ and every $\omega\in U\cup T\cup \{\tau_0,\tau_0^{-1}\}\cup\{s_{\alpha},\tau_{\alpha}\,|\,\alpha\in\Sigma\}$ the element $\vartheta'(f_x)\f_\omega$ is in the image of $\vartheta'$.
So, from now on,
we fix $x\in G$ and we consider the unique decomposition $f_x=f_{u_1}f_tf_{\tau_0}^if_{w_1}f_\tau f_{w_2}f_{u_2}$ of $f_x$ of the form (\ref{prodottogeneratori2}).
\begin{enumerate}[$\bullet$]
\item If $\tilde u\in U$, by 1 of definition \ref{defA} and by lemma \ref{lemmaprodotto2} we obtain  
$\vartheta'(f_x)\f_{\tilde{u}}=\f_{u_1}\f_t\f_{\tau_0}^i\f_{w_1}\f_\tau\f_{w_2}\f_{u_2\tilde{u}}
=\vartheta'(f_{u_1}f_tf_{\tau_0}^if_{w_1}f_\tau f_{w_2}f_{u_2\tilde{u}})$.
\item If $\tilde t\in T$, by 1 and 3 of definition \ref{defA}, by remark \ref{rmktau0} and by lemma \ref{lemmaprodotto2} we obtain  
$\vartheta'(f_x)\f_{\tilde{t}}
=\f_{u_1} \f_{t(\tau_0^iw_1\tau w_2)\tilde{t}(\tau_0^{i}w_1\tau w_2)^{-1}}\f_{\tau_0}^i\f_{w_1}\f_\tau \f_{w_2}\f_{\tilde{t}^{-1}u_2\tilde{t}}
=\vartheta'(f_{u_1} f_{t(\tau_0^iw_1\tau w_2)\tilde{t}(\tau_0^{i}w_1\tau w_2)^{-1}}f_{\tau_0}^if_{w_1}f_\tau f_{w_2}f_{\tilde{t}^{-1}u_2\tilde{t}})$.
\item By remark \ref{rmktau0} 
and by lemma \ref{lemmaprodotto2} we obtain
$\vartheta'(f_x)\f_{\tau_0^{\,\varepsilon}}
=\f_{u_1}\f_t\f_{\tau_0}^{i+\varepsilon}\f_{w_1}\f_\tau \f_{w_2}
\f_{\tau_0^{-\varepsilon}u_2\tau_0^{+\varepsilon}}
=\vartheta'(f_{u_1}f_tf_{\tau_0}^{i+\varepsilon}f_{w_1}f_\tau f_{w_2}
f_{\tau_0^{-\varepsilon}u_2\tau_0^{+\varepsilon}})$ for $\varepsilon\in\{1,-1\}$.
\item Let $\alpha\in\Sigma$ and $s=s_{\alpha}$. 
Since $\alpha$ is simple, 
$U_{\{\alpha\}}^+=\prod_{\alpha'\in\mathbf{\Phi}^+\setminus\{\alpha\}}U_{\alpha'}$ is a group and we have $U\cap w_2^{-1}Uw_2=(U_{\alpha}\cap w_2^{-1}Uw_2) (U_{\{\alpha\}}^+\cap w_2^{-1}Uw_2)$.  
We take $u_2=v_1v_2$ according to this decomposition. 
Then $sv_2s\in U$, $sv_1s\in U_{-\alpha}$ and by 1 of definition \ref{defA} we have
$\vartheta'(f_x)\f_{s}=\f_{u_1}\f_t\f_{\tau_0}^i\f_{w_1}\f_\tau \f_{w_2s}\f_{sv_1s}\f_{sv_2s}$.
If $v_1\in K^1$ then
$\vartheta'(f_x)\f_{s}=\vartheta'(f_{u_1}f_tf_{\tau_0}^if_{w_1}f_\tau f_{w_2s}f_{s u_2s})$ by lemma \ref{lemmaprodotto2}.
From now on we suppose $v_1\notin K^1$. 
Using remark \ref{rmkscomposizione}, it is easy to show that there exist $k_1,k_2\in K^1$, $v_3,v_4\in U_{\alpha}$ and $t'\in T$ such that $sv_1s=k_1v_3st'v_4k_2$ and so by 1 of definition \ref{defA} we obtain 
$\vartheta'(f_x)\f_{s}=
\f_{u_1}\f_t\f_{\tau_0}^i\f_{w_1}\f_\tau \f_{w_2s}\f_{v_3st'v_4}\f_{sv_2s}
=\f_{u_1}\f_t\f_{\tau_0}^i\f_{w_1}\f_{\tau}\f_{w_2sv_3sw_2^{-1}}\f_{w_2t'w_2^{-1}}\f_{w_2}\f_{v_4s v_2s}$.
Moreover $v_1\notin K^1$ implies that $U_\alpha\subset w_2^{-1}Uw_2$ and so $\alpha\in w_2^{-1}\mathbf{\Phi}^+$. 
We can distinguish three cases.
\begin{enumerate}[$\star$]
	\item If $\alpha\in w_2^{-1}\mathbf{\Psi}_{P(\tau)}^+$ then $w_2s v_3s w_2^{-1}\in U_{P(\tau)}^-$ 
	and we obtain 
\begin{align*}
\vartheta'(f_x)\f_{s}&
=\f_{u_1}\f_t\f_{\tau_0}^i\f_{w_1}\f_{\tau}\f_{w_2sv_3sw_2^{-1}}\f_{w_2t'w_2^{-1}}\f_{w_2}\f_{v_4s v_2s}\\
\text{\scriptsize{(lemma \ref{coroldefA})}}	
&=\f_{u_1}\f_t\f_{\tau_0}^i\f_{w_1}\f_{\tau}\f_{w_2t'w_2^{-1}}\f_{w_2}\f_{v_4s v_2s} \\
\text{\scriptsize{(1,3 of def. \ref{defA}, remark \ref{rmktau0})}}
&=\f_{u_1}\f_{t\tau_0^iw_1\tau w_2t'(\tau_0^iw_1\tau w_2)^{-1}}\f_{\tau_0}^i\f_{w_1}\f_\tau\f_{w_2}\f_{v_4s v_2s}\\
\text{\scriptsize{(lemma \ref{lemmaprodotto2})}}
&=\vartheta'( f_{u_1} f_{t\tau_0^iw_1\tau w_2t'(\tau_0^iw_1\tau w_2)^{-1}} f_{\tau_0}^i f_{w_1} f_\tau f_{w_2} f_{v_4s v_2s}).
\end{align*}
	\item 
	If $\alpha\in w_2^{-1}\mathbf{\Phi}_{P(\tau)}^+\cap w_2^{-1}w_1^{-1}\mathbf{\Phi}^-$ 
	then $w_2sv_3sw_2^{-1}\in M_{P(\tau)}^-$ and $w_1\tau w_2s v_3(w_1\tau w_2s)^{-1}\in U$. 
We obtain 
\begin{align*}
\hspace{-1cm}
\vartheta'(f_x)\f_{s}&
=\f_{u_1}\f_t\f_{\tau_0}^i\f_{w_1}\f_{\tau}\f_{w_2sv_3sw_2^{-1}}\f_{w_2t'w_2^{-1}}\f_{w_2}\f_{v_4s v_2s}\\
\text{\scriptsize{(lemma \ref{coroldefA})}} 
&=\f_{u_1}\f_t\f_{\tau_0}^i\f_{w_1}\f_{\tau w_2sv_3sw_2^{-1}\tau^{-1}}\f_{\tau}\f_{w_2t'w_2^{-1}}\f_{w_2}\f_{v_4s v_2s}\\
\text{\scriptsize{(1 def. \ref{defA}, rmk \ref{rmktau0})}}
&=\f_{u_1(t\tau_0^iw_1\tau w_2s) v_3 (t\tau_0^iw_1\tau w_2s)^{-1}}\f_t\f_{\tau_0}^i\f_{w_1}\f_{\tau}\f_{w_2t'w_2^{-1}}\f_{w_2}\f_{v_4s v_2s}\\
\text{\scriptsize{(1,3 def. \ref{defA}, rmk \ref{rmktau0})}}
&=\f_{u_1t\tau_0^iw_1\tau w_2s v_3(t\tau_0^iw_1\tau w_2s)^{-1}}\f_{t\tau_0^iw_1\tau w_2 t'(\tau_0^iw_1\tau w_2)^{-1}}\f_{\tau_0}^i\f_{w_1}
\f_\tau\f_{w_2}\f_{v_4s v_2s}\\
\text{\scriptsize{(lemma \ref{lemmaprodotto2})}}
&=\vartheta'(  f_{u_1t\tau_0^iw_1\tau w_2s v_3(t\tau_0^iw_1\tau w_2s)^{-1}} f_{t\tau_0^iw_1\tau w_2 t'(\tau_0^iw_1\tau w_2)^{-1}} f_{\tau_0}^i f_{w_1} f_\tau f_{w_2} f_{v_4s v_2s}).
\end{align*}
\item If $\alpha\in w_2^{-1}\mathbf{\Phi}_{P(\tau)}^+\cap w_2^{-1}w_1^{-1}\mathbf{\Phi}^+$ then
$w_2v_1w_2^{-1}\in M_{P(\tau)}^-$ and $w_1\tau w_2v_1 (w_1\tau w_2)^{-1}\in U$. 
We obtain
\begin{align*}
\vartheta'(f_x)\f_{s}&
=\f_{u_1}\f_t\f_{\tau_0}^i\f_{w_1}\f_\tau \f_{w_2s}\f_{sv_1s}\f_{sv_2s}\\
\text{\scriptsize{(1 of def. \ref{defA})}} 
&=\f_{u_1}\f_t\f_{\tau_0}^i\f_{w_1}\f_\tau \f_{w_2v_1w_2^{-1}}\f_{w_2s}\f_{sv_2s}\\
\text{\scriptsize{(lemma \ref{coroldefA})}}
&=\f_{u_1}\f_t\f_{\tau_0}^i\f_{w_1}\f_{\tau w_2v_1w_2^{-1}\tau^{-1}}\f_\tau \f_{w_2s}\f_{sv_2s}\\
\text{\scriptsize{(1 of def. \ref{defA}, remark \ref{rmktau0})}}
&=\f_{u_1t\tau_0^iw_1\tau w_2v_1 (t\tau_0^iw_1\tau w_2)^{-1}}\f_t\f_{\tau_0}^i\f_{w_1}\f_\tau\f_{w_2s}\f_{s v_2s}\\
\text{\scriptsize{(lemma \ref{lemmaprodotto2})}}
&= \vartheta'(f_{u_1t\tau_0^iw_1\tau w_2v_1 (t\tau_0^iw_1\tau w_2)^{-1}} f_t f_{\tau_0}^i f_{w_1} f_\tau f_{w_2s} f_{s v_2s}).
\end{align*}
\end{enumerate}
So we have proved that $\vartheta'(f_x)\f_{s_\alpha}$ is in the image of $\vartheta'$ for every $\alpha\in \Sigma$.
\item Let $\alpha\in \Sigma$ and $u_2=v_1v_2$ be according to the decomposition $U=M_{\widehat\alpha}^+U_{\widehat\alpha}^+$. 
By 4 and 5 of definition \ref{defA} we have 
$\vartheta'(f_x)\f_{\tau_\alpha}=\f_{u_1}\f_t\f_{\tau_0}^i\f_{w_1}\f_{\tau}\f_{w_2}\f_{\tau_\alpha}\f_{\tau_{\alpha}^{-1}v_1\tau_{\alpha}}$. 
If $w$ is the element of minimal length in $w_2W_{\widehat\alpha}\in W/W_{\widehat\alpha}$ then by proposition \ref{proplongmin2} it is of minimal length also in $W_{P(\tau)}w$. 
If we set $P=P(w,\alpha)$ and $Q=Q(w,\alpha)$ then by lemma \ref{lemmaprodottotau0} we obtain 
$\tau\tau_P^{-1}\in\bm\Delta$ and so
\begin{align*}
\vartheta'(f_x)\f_{\tau_\alpha}&
=\f_{u_1}\f_t\f_{\tau_0}^i\f_{w_1}\f_{\tau}\f_{w_2}\f_{\tau_\alpha}\f_{\tau_{\alpha}^{-1}v_1\tau_{\alpha}}\\
\text{\scriptsize{(1,8 of def. \ref{defA})}}
&=\f_{u_1}\f_t\f_{\tau_0}^i\f_{w_1}\f_{\tau\tau_P^{-1}}\f_{\tau_P}\f_{w}\f_{w^{-1}w_2}\f_{\tau_\alpha}\f_{\tau_{\alpha}^{-1}v_1\tau_{\alpha}}\\
\text{\scriptsize{(7 of def. \ref{defA})}}
&=\f_{u_1}\f_t\f_{\tau_0}^i\f_{w_1}\f_{\tau\tau_P^{-1}}\f_{\tau_P}\f_{w}\f_{\tau_\alpha}\f_{w^{-1}w_2}\f_{\tau_{\alpha}^{-1}v_1\tau_{\alpha}}\\
\text{\scriptsize{(1,9 of def. \ref{defA})}}
&=q^{\ell(w)}\sum_{u}  
\f_{u_1}\f_t\f_{\tau_0}^i\f_{w_1}\f_{\tau\tau_P^{-1}} \f_{\tau_Q} \f_u   \f_{w_2}\f_{\tau_{\alpha}^{-1}v_1\tau_{\alpha}}\\
\text{\scriptsize{(8 of def. \ref{defA})}}
&=q^{\ell(w)}\sum_{u}
\Big(\f_{u_1}\f_t\f_{\tau_0}^i\f_{w_1}\f_{\tau\tau_P^{-1}\tau_Q} \f_u\Big) \f_{w_2}\f_{\tau_{\alpha}^{-1}v_1\tau_{\alpha}}\\
\text{\scriptsize{(lemma \ref{lemmaprodotto2})}}
&=q^{\ell(w)}\sum_{u} 
\vartheta'(f_{u_1}f_tf_{\tau_0}^if_{w_1}f_{\tau\tau_P^{-1}\tau_Q} f_u) \f_{w_2}\f_{\tau_{\alpha}^{-1}v_1\tau_{\alpha}}.
\end{align*}
By the above calculations, the element $\vartheta'(f_{u_1}f_tf_{\tau_0}^if_{w_1}f_{\tau\tau_P^{-1}\tau_Q} f_u) \f_{w_2}\f_{\tau_{\alpha}^{-1}v_1\tau_{\alpha}}$ is in the image of $\vartheta'$ for every $u\in U\cap wUw^{-1}$ and so $\vartheta'(f_x)\f_{\tau_\alpha}$ too.
\end{enumerate}
So we can conclude that $\vartheta'$ is a surjective homomorphism of $\Z$-modules.
\end{proof}

We now have all the tools necessary to prove the main result.

\begin{corol}\label{isomalgebre}
The algebras $\mathscr{A}$ and $\mathscr{H}(G,K^1)$ are isomorphic.
\end{corol}
\begin{proof}
We have constructed a surjective homomorphism of algebras $\vartheta:\mathscr{A}\longrightarrow \mathscr{H}(G,K^1)$ and a surjective homomorphism of $\Z$-modules $\vartheta':\mathscr{H}(G,K^1)\longrightarrow \mathscr{A}$ such that $\vartheta(\vartheta'(\Phi))=\Phi$ 
for every $\Phi\in\mathscr{H}(G,K^1)$. 
This implies $\vartheta'(\vartheta(\vartheta'(\Phi)))=\vartheta'(\Phi)$ for every $\Phi\in\mathscr{H}(G,K^1)$ and so 
$\vartheta'(\vartheta(\mathbf{a}))=\mathbf{a}$ for every $\mathbf{a}\in\mathscr{A}$
Thus $\vartheta$ is an isomorphism of algebras whose inverse is $\vartheta'$.
\end{proof}

\section{Level-$0$ representations}
In this section we turn to the study of the category of smooth representations of a direct product of inner forms of general linear groups 
over non-archimedean locally compact fields. 
We prove that its level-$0$ subcategory is equivalent to the category of modules over the algebra described in the previous section.

\smallskip
Let us use notations of the beginning of section 2. 
We denote $G=G_1\times\cdots\times G_r$ and $K^1=K_1^1\times\cdots\times K_r^1$.
Let $R$ be a unitary commutative ring such that $p\in R^{\times}$, let $\mathscr{H}_R(G,K^1)=\mathscr{H}(G,K^1)\otimes_\Z R$ (see remark \ref{remR}) and let $\mathscr{R}_R(G)$ be the category of smooth representations of $G$ over $R$. 
From now on all representations that we consider are smooth. 
If $(\pi,V)$ is a representation of $G$ we denote by $V^{K^1}$ the set $\{v\in V\,|\,\pi(k)v=v \text{ for every }k\in K^1\}$ of $K^1$-invariant vectors of $V$.

\begin{defin}
A representation $(\pi,V)$ of $G$ is called \emph{a level-$0$ representation} if it is generated by its $K^1$-invariant vectors. 
We denote by $\mathscr{R}^0_R(G)$ the full subcategory of $\mathscr{R}_R(G)$ of level-$0$ representations, which we call \emph{level-$0$ subcategory} of $\mathscr{R}_R(G)$. 
\end{defin}

Since $p$ is invertible in $R$ we can choose an Haar measure $dg$ on $G$ with values in $R$ such that $\int_{K^1}dg=1$. 
Then we can define the global Hecke algebra $\mathscr{H}_R(G)$ (see \cite{Vig2} I.3.1) as the $R$-algebra of locally constant and compactly supported functions $f:G\longrightarrow R$ endowed with the convolution product given by $(f_1*f_2)(x)=\int_{G}f_1(g)f_2(g^{-1}x)dg$ for every $x\in G$. 
We recall that $\mathscr{R}_R(G)$ is equivalent to the category of (left) modules over $\mathscr{H}_R(G)$ (see \cite{Vig2} I.4.4). 

\smallskip
Let us use section I.6 of \cite{Vig2} with $A=\mathscr{H}_R(G)$ and $e$ the characteristic function of $K^1$. 
The algebra $eAe$ is the algebra $\mathscr{H}_R(G,K^1)$ and, thanks to equivalence of categories above, the functor $V\longmapsto eV$ becomes the functor $\inv_{K^1}$ of $K^1$-invariants from $\mathscr{R}_R(G)$ to the category of right modules over $\mathscr{H}_R(G,K^1)$ defined by 
$\inv_{K^1}(V)=V^{K^1}$ with $\Phi\in\mathscr{H}_R(G,K^1)$ acting on $v\in V^{K^1}$ by $\sum_{x\in K^1\bs G}\Phi(x)\pi(x^{-1})v$. 
Now, since $\mathscr{H}_R(G)$ is an algebra with enough idempotents (see \cite{Vig2} I.3.2), hypothesis of I.6.6 are satisfied and then we obtain the following result.

\begin{teor}\label{teoremaequivalenzacategorie}
If every non-zero subrepresentation $W$ of any object $V$ in $\mathscr{R}^0_R(G)$ verify $W^{K^1}\neq 0$ then $\inv_{K^1}$ is an equivalence between $\mathscr{R}^0_R(G)$ and the category of right modules over $\mathscr{H}_R(G,K^1)$.
\end{teor}

By proposition 6.3 of \cite{Dat} (see also \cite{Vig2} II.5 when $R$ is a field) there exists a decomposition $\mathscr{R}_R(G)=\mathscr{R}^0_R(G)\oplus \mathscr{R}^{>0}_R(G)$ that means that:
\begin{itemize}
\item for every $R$-representation $V$ of $G$ there exist an object $V_1$ of $\mathscr{R}^0_R(G)$ and an object $V_2$ of $\mathscr{R}^{>0}_R(G)$ such that $V=V_1\oplus V_2$;
\item for every object $V_1$ of $\mathscr{R}^0_R(G)$ and every object $V_2$ of $\mathscr{R}^{>0}_R(G)$ we have $\Hom_G(V_1,V_2)=0$.
\end{itemize}
This fact implies that level-$0$ representations are exactly those whose all irreducible subquotients have non-zero $K^1$-invariant vectors. 
Hence, if $W$ is a subrepresentation of an object of $\mathscr{R}^0_R(G)$ then $W^{K^1}$ must be non-zero and so theorem \ref{teoremaequivalenzacategorie} implies the following.

\begin{corol}\label{corolequivalenza}
The level-$0$ subcategory $\mathscr{R}^0_R(G)$ of $\mathscr{R}_R(G)$ is equivalent to the category of right modules over $\mathscr{H}_R(G,K^1)$.
\end{corol} 

\subsubsection*{Acknowledgements}
This work is part of the PhD thesis of the author and he wants to thank his supervisor, Vincent Sécherre, for his support and his comments on this paper. 
This research did not receive any specific grant from funding agencies in the public, commercial,
or not-for-profit sectors.

\bibliographystyle{amsplain}
\bibliography{paper1}

\providecommand{\bysame}{\leavevmode\hbox to3em{\hrulefill}\thinspace}
\providecommand{\MR}{\relax\ifhmode\unskip\space\fi MR }
\providecommand{\MRhref}[2]{%
  \href{http://www.ams.org/mathscinet-getitem?mr=#1}{#2}
}
\providecommand{\href}[2]{#2}
\begin{thebibliography}{10}

\bibitem{Bern}
Joseph Bernstein, \emph{Le "centre" de {B}ernstein}, Représentations des
  groupes réductifs sur un corps local, Travaux en cours. Hermann, Paris
  (1984), 1--32, Edited by P. Deligne.

\bibitem{Bourb1}
Nicolas Bourbaki, \emph{Lie groups and {L}ie algebras. {C}hapters 4-6},
  Springer-Verlag, Berlin, 2002.

\bibitem{BK2}
Colin~J. Bushnell and Philip~C. Kutzko, \emph{Smooth representations of
  reductive $p$-adic groups: structure theory via types}, Proc. London Math.
  Soc. (3) \textbf{77} (1998), 582--634,
  http://dx.doi.org/10.1112/S0024611598000574.

\bibitem{BK3}
\bysame, \emph{Semisimple types in {$GL_n$}}, Compositio Math. \textbf{119}
  (1999), no.~1, 53--97, http://dx.doi.org/10.1023/A:1001773929735.

\bibitem{Chin}
Gianmarco Chinello, \emph{Représentations $\ell$-modulaires des groupes
  $p$-adiques. {D}écomposition en blocs de la catégorie des représentations
  lisses de {$GL_m(D)$}, groupe métaplectique et représentation de {W}eil},
  Ph.D. thesis, Université de Versailles St-Quentin-en-Yvelines, 2015.

\bibitem{Dat}
Jean-Francois Dat, \emph{Finitude pour les représentations lisses de groupes
  $p$-adiques}, J. Inst. Math. Jussieu \textbf{8} (2009), 261--333,
  http://dx.doi.org/10.1017/S1474748008000054.

\bibitem{Dat3}
\bysame, \emph{Théorie de {L}ubin-{T}ate non-abélienne $\ell$-entière}, Duke
  Math. J. \textbf{161} (2012), no.~6, 951--1010,
  http://dx.doi.org/10.1215/00127094-1548425.

\bibitem{Gui}
David-Alexandre Guiraud, \emph{On semisimple $l$-modular {B}ernstein-blocks of
  a $p$-adic general linear group}, J. Number Theory \textbf{133} (2013),
  3524--3548, http://dx.doi.org/10.1016/j.jnt.2013.04.012.

\bibitem{Helm}
David Helm, \emph{The {B}ernstein center of the category of smooth
  {$W(k)[GL_n(F)]$}-modules}, Forum of Mathematics, Sigma \textbf{4} (2012),
  1--98, http://dx.doi.org/10.1017/fms.2016.10.

\bibitem{Krieg}
Aloys Krieg, \emph{Hecke algebras}, vol.~87, Mem., no. 435, Amer. Math. Soc.,
  1990, http://dx.doi.org/10.1090/memo/0435.

\bibitem{MS}
Alberto Minguez and Vincent Sécherre, \emph{Types modulo $\ell$ pour les formes
  intérieures de {$GL_n$} sur un corps local non archimédien}, Proc. London
  Math. Soc. (3) \textbf{109} (2014), 823--891, With an appendix by Vincent
  Sécherre et Shaun Stevens. http://dx.doi.org/10.1112/plms/pdu020.

\bibitem{Morris}
Lawrence Morris, \emph{Tamely ramified intertwining algebras}, Inventiones
  mathematicae \textbf{114} (1993), 1--54,
  http://dx.doi.org/10.1007/BF01232662.

\bibitem{SeSt3}
Vincent Sécherre and Shaun Stevens, \emph{Smooth representations of {$GL_m(D)$}
  {VI}: semisimple types}, Int. Math. Res. Not. IMRN \textbf{13} (2011),
  2994--3039, http://dx.doi.org/10.1093/imrn/rnr122.

\bibitem{SeSt1}
\bysame, \emph{Block decomposition of the category of $\ell$-modular smooth
  representations of {$GL_n(F)$} and its inner forms}, Ann. Scient. Éc. Norm.
  Sup. \textbf{49} (2016), no.~3, 669--709.

\bibitem{Vig2}
Marie-France Vign{\'e}ras, \emph{Repr{\'e}sentations $l$-modulaires d'un groupe
  r{\'e}ductif $p$-adique avec $l\neq p$}, Progress in Mathematics, vol. 137,
  Birkh{\"a}user Boston, 1996.

\bibitem{Vig1}
\bysame, \emph{Induced {$R$}-representations of $p$-adic reductive groups},
  Selecta Math. \textbf{4} (1998), 549--623,
  http://dx.doi.org/10.1007/s000290050040.

\end{thebibliography}

\end{document}